\documentclass[journal,twoside,web]{ieeecolor}
\usepackage{generic}
 \def\BibTeX{{\rm B\kern-.05em{\sc i\kern-.025em b}\kern-.08em
     T\kern-.1667em\lower.7ex\hbox{E}\kern-.125emX}}
\usepackage{cite}
\usepackage{amsmath,amssymb,amsfonts}
\usepackage{algorithmic}
\usepackage{graphicx}
\usepackage{algorithm,algorithmic}
\usepackage{hyperref}
\usepackage{textcomp}
\usepackage{cite}
\usepackage{float}
\usepackage{amsmath,amssymb,amsfonts}
\usepackage{algorithmic}
\usepackage{mathrsfs}
\usepackage{graphicx}
\usepackage{textcomp}
\usepackage[amsmath,thmmarks]{ntheorem}
\usepackage{hyperref}
\usepackage{enumerate}
\hypersetup{pdfborder={0 0 0},}

\usepackage{algorithm}
\usepackage{algorithmic}
 \theoremheaderfont{\normalfont\bfseries} 
\theorembodyfont{\normalfont\itshape}    
\theoremseparator{.}                     

\newtheorem{remark}{Remark}
\newtheorem{theorem}{Theorem}
\newtheorem{lemma}{Lemma}

\newtheorem{assumption}{Assumption}

\newcommand{\ADD}[1]{\textcolor{black}{{#1}}}

\def\BibTeX{{\rm B\kern-.05em{\sc i\kern-.025em b}\kern-.08em
    T\kern-.1667em\lower.7ex\hbox{E}\kern-.125emX}}
\begin{document}
  \title{Distributed and Decentralized Optimization Algorithms via Consensus ALADIN}
\author{Xu Du, \IEEEmembership{Member, IEEE}, Jingzhe Wang, \IEEEmembership{Student Member, IEEE}, Karl~H.~Johansson, \IEEEmembership{Fellow, IEEE}, and Apostolos I. Rikos$^*$, \IEEEmembership{Member, IEEE}
\thanks{$^*$Corresponding author.}
	\thanks{Xu Du and Apostolos I. Rikos are with the Artificial Intelligence Thrust of the Information Hub, The Hong Kong University of Science and Technology (Guangzhou), Guangzhou, China. 
    Apostolos I. Rikos is also affiliated with the Department of Computer Science and Engineering, The Hong Kong University of Science and Technology, Clear Water Bay, Hong Kong, China. E-mails: {\tt~\{michaelxudu;apostolosr\}@hkust-gz.edu.cn}. 
            }
            \thanks{Jingzhe Wang is with School of Computing and Information, University of Pittsburgh, Pittsburgh, PA, USA. E-mail: \texttt{jiw148@pitt.edu}.}
            \thanks{Karl H.~Johansson is with the Division of Decision and Control Systems, KTH Royal Institute of Technology, SE-100 44 Stockholm, Sweden. 
    He is also affiliated with Digital Futures, SE-100 44 Stockholm, Sweden. 
    E-mail:{\tt~kallej@kth.se}.
            }
            \thanks{Preliminary results of this work were presented at the 2025 American Control Conference \cite{Du2025ACC} and the 2025 IEEE Conference on Decision and Control \cite{Du2025CDCA}. The current version of our paper includes: (i) the complete proof of the algorithm proposed in \cite{Du2025CDCA}; (ii) second‑order variants of the algorithms in \cite{Du2025CDCA} for non‑convex problems; and (iii) extended comparisons with algorithms from the literature.}
\thanks{The work of X.D. and A.I.R. was supported by the Guangzhou-HKUST(GZ) Joint Funding Scheme (Grant No. 2025A03J3960). The work of A.I.R. was also supported by the Guangdong Provincial Project (Grant No. 2024QN11G109).} 
}

\maketitle

\begin{abstract}
Distributed optimization has found widespread applications in smart grids, optimal control, and machine learning. 
This paper studies distributed consensus optimization. 
We extend the Augmented Lagrangian-based Alternating Direction Inexact Newton (ALADIN) framework to propose Consensus ALADIN (C-ALADIN) with a central coordinator, which directly handles consensus constraints. 
Our C-ALADIN algorithm admits both a first-order variant and a second-order variant that employs a Hessian approximation, avoiding direct transmission of second-order information while preserving fast local convergence.
We then develop a decentralized version of C-ALADIN that operates over directed graphs with quantized communication, using a finite-time coordination protocol. 
For both versions, we establish global convergence guarantees for convex problems and local convergence guarantees for non-convex problems. 
For the decentralized case, the iterates converge to a neighborhood of the optimum determined by the quantization level. 
Numerical results demonstrate that our methods retain fast convergence while substantially reducing communication and computational costs compared to existing decentralized approaches.
\end{abstract}


\section{Introduction}
\label{sec:introduction}

Distributed and decentralized optimization have attracted significant attention, driven by advances in distributed control systems \cite{stomberg2025decentralized,jiang2026distributed,wu2025time}, federated learning \cite{zhou2026preconditioned,zhou2023federated}, and distributed power systems \cite{steven2026distributed,Du2019}. 
The growing need to solve large-scale problems with massive datasets and heterogeneous objectives motivates the use of distributed and decentralized optimization algorithms.

In distributed optimization, data are partitioned across multiple agents. 
Each agent solves a local subproblem and exchanges information either with a central coordinator or directly with its neighboring agents. 
When no coordinator is present and communication relies solely on peer-to-peer interactions, the setting is specifically referred to as decentralized optimization. Distributed optimization problems are typically categorized into two main classes: (i) distributed resource allocation optimization, and (ii) distributed consensus optimization \cite{Du2025ACC}. 
In the first class, the objective function has a finite-sum structure, and local variables are coupled through affine equality or inequality constraints, potentially with heterogeneous dimensions \cite{2024_doostmohammadian_rikos_Johansson_survey}. In the second class, problems also exhibit a sum-structured objective, but the local variables are required to reach full consensus \cite{boyd2011distributed}. 
From an algorithmic perspective, as surveyed in \cite{boyd2007notes,2024_doostmohammadian_rikos_Johansson_survey}, distributed and decentralized optimization methods can be divided into two types: (i) primal decomposition based algorithms and (ii) dual decomposition based algorithms. 
Primal decomposition based optimization includes zeroth-order \cite{xu2024quantized}, first-order \cite{nedic2009distributed,Shi2015,ioffe2015batch,mcmahan2017communication,wu2026nesterov,dist_structure4_lihuaxie,pushi,jiang2022fast}, and second-order methods \cite{wang2018giant,zhang2021newton,dinc2024incremental}. In contrast, dual decomposition based approaches build upon fundamental frameworks such as Dantzig-Wolfe decomposition \cite{dantzig1960decomposition}, the Alternating Direction Method of Multipliers (ADMM) \cite{boyd2011distributed}, and the Augmented Lagrangian Alternating Direction Inexact Newton (ALADIN) method \cite{Houska2016,Houska2021}. Dual decomposition generally achieves faster convergence and higher solution accuracy than primal methods \cite[Sec.~I]{ling2015dlm}. This work therefore focuses on dual decomposition based approaches for distributed and decentralized consensus optimization, specifically on the ALADIN method \cite{Houska2016,Houska2021}. 
This algorithm has demonstrated favorable convergence performance in distributed resource allocation optimization. 
However, its application to consensus optimization faces structural challenges, as we discuss next. 
Extending ALADIN to this setting therefore remains largely unexplored in the literature.

\vspace{-3mm}
\subsection{Literature Review.} 
The ALADIN algorithm \cite{Houska2016} was originally introduced to improve the numerical convergence speed of ADMM \cite{boyd2011distributed,notarnicola2025passivity,du2026affine} for distributed resource allocation optimization, specifically its parallelizable primal splitting variant \cite{wang2013solving}, \cite[Algorithm~1]{Houska2016}, and to strengthen its convergence guarantees for non-convex problems. It achieves this by integrating the coupled Sequential Quadratic Programming (SQP) \cite{Nocedal2006} into the augmented Lagrangian framework \cite{Houska2016,Houska2021,Engelmann2019,Engelmann2020,Du2019,wang2025aladin,Du2023_Arxiv,Du2025ACC,du2025convergence,Du2023B,han2026mix,Shi2022}.
In convergence theory, ALADIN-type algorithms achieve local convergence guarantees for smooth non-convex problems and global convergence guarantees for convex problems \cite{Houska2021,Du2025ACC}.
ALADIN has subsequently been applied to model predictive control (MPC) \cite{Shi2022}, moving horizon estimation (MHE) \cite{wu2025time,wang2026lightweight}, and distributed optimal power flow (OPF) \cite{Engelmann2019,Du2019}. In each case, the problem suffers from the curse of dimensionality: for MPC and MHE, due to a long horizon, and for OPF, due to a large number of buses. By recasting these problems into distributed formulations via time-splitting for MPC/MHE or spatial decomposition for OPF, and solving them with ALADIN, the dimensionality burden is effectively alleviated. 
Inspired by ALADIN, the PaDOA algorithm \cite{Villanueva2021} was developed for distributed mixed-integer programming, integrating outer approximation into the distributed optimization framework, achieving global convergence for convex problems and finding applications in temperature control systems. 
Despite the success of ALADIN in distributed resource allocation optimization problems, its application to distributed consensus optimization faces several challenges: (i) ALADIN requires each variable to possess local second-order information, yet the global consensus variable has no associated objective function and consequently admits no Hessian, leading to a structural mismatch; 
(ii) existing ALADIN implementations typically rely on a central coordinator or assume undirected communication, which is not always feasible in many real-world applications where directed graphs are prevalent; and (iii) ALADIN relies on real-valued exchange of Hessians and gradients, which limits applicability in environments with limited communication bandwidth. 

To address the first challenge, a coupling matrix scheme was proposed in \cite[Sec.~12]{aadhithya2023learning} that reformulates consensus optimization into a resource allocation form, but this requires manual intervention, leaving direct consensus solving within ALADIN underexplored. 
To address the second challenge, prior work \cite{Engelmann2020} employs decentralized ADMM to solve the coupled QP in ALADIN. 
In an attempt to tackle decentralized consensus optimization directly, the GIANT algorithm \cite{wang2018giant} and its variants employ global gradient aggregation with exact local Hessian information. Nevertheless, they have been primarily studied only for convex problems on undirected graphs.
The ADMM-type algorithm in \cite{rikos2023asynchronous} simultaneously addresses all aforementioned challenges 
through the integration of Consensus ADMM with a decentralized quantized communication scheme; however, it remains limited to convex problems and is not based on the ALADIN framework.


Based on the above literature review, three critical gaps remain for ALADIN‑type algorithms in the context of distributed consensus optimization:
(i) tailoring ALADIN specifically for consensus constraints (rather than resource allocation);
(ii) decentralization over directed graphs without a central coordinator; and
(iii) quantized communication to reduce bandwidth.
Although some prior works address subsets of these issues (e.g., quantization for ADMM, decentralized QP solvers for resource allocation), no existing ALADIN‑based method addresses all three in a unified manner, particularly for non‑convex problems. 
To the best of our knowledge, apart from our recent works \cite{Du2025ACC,Du2025CDCA}, this gap has not been systematically addressed.

\vspace{-3mm}
\subsection{Main Contributions.} 
Motivated by the aforementioned open questions, we propose Consensus ALADIN (C-ALADIN) and its extensions built upon the ALADIN framework. 
Our main contributions are as follows.
\noindent
\\ \textbullet \ \textbf{C-ALADIN framework.} 
To tailor ALADIN specifically for distributed consensus optimization, inspired by Consensus ADMM \cite[Ch.~7]{boyd2011distributed}, we introduce C-ALADIN (as shown in \eqref{alg:ALADIN}) as a more natural solution that replaces the generic coupled quadratic program (QP) used for information coordination in ALADIN with a consensus QP.
\noindent
\\ \textbullet \ \textbf{Second-order C-ALADIN.} 
To reduce communication and computational complexity, second-order C-ALADIN is introduced. 
The coordinator employs a Broyden-Fletcher-Goldfarb-Shanno (BFGS)-based Hessian approximation for each agent, avoiding direct transmission of second-order information (Algorithm~\ref{alg:BFGS ALADIN2}). 
A closed‑form solution for the large‑scale consensus QP further reduces computational complexity. 
The resulting algorithm provides local convergence guarantees for non-convex consensus problems (Theorem~\ref{them: local convergence}) and, under constant Hessian assumptions, ensures global convergence for convex problems (Theorem~\ref{The: 1} and \ref{theorem 2}). 
\noindent
\\ \textbullet \ 
\textbf{First-order C-ALADIN.} 
For scenarios where the Hessian approximation matrices are difficult to evaluate, we introduce the first-order C-ALADIN (as shown in \eqref{eq: RCA}) by replacing the Hessian in the BFGS-based C-ALADIN with a suitably scaled identity matrix. 
This corresponds to the constant Hessian case $B_i = \rho I$ in Second-order C-ALADIN, and therefore inherits the same global convergence guarantees for convex problems (Theorems~\ref{The: 1} and~\ref{theorem 2}).
\noindent
\\ \textbullet \ 
\textbf{Decentralized first-order C-ALADIN with quantized communication.} 
For scenarios where agents communicate over directed decentralized networks,
building upon the first-order C-ALADIN, we introduce a decentralized variant for directed communication graphs (as shown in \eqref{eq: QuDRC-ALADIN}) that incorporates a quantized finite-time consensus protocol (Algorithm~\ref{alg:QuAS}). 
This variant enables agents to exchange quantized messages, reducing communication overhead. 
We prove that under convex local cost functions, it converges globally to a quantization-dependent neighborhood of the optimal solution (see Theorem~\ref{the: convergence}). 
\noindent
\\ \textbullet \ 
\textbf{Decentralized second-order C-ALADIN with quantized communication.} 
For scenarios involving non-convex consensus problems over directed networks with quantized communication,
we extend second-order C-ALADIN to two algorithmic variants (Algorithm~\ref{alg:Bilevel Consensus ALADIN} and Algorithm~\ref{alg: Decentralized Reduced BFGS Consensus ALADIN}). 
The first adopts a bilevel structure, solving the consensus QP via the decentralized first-order C-ALADIN (as shown in \eqref{eq: QuDRC-ALADIN}) to achieve full decentralization. The second employs Algorithm~\ref{alg:QuAS} (a decentralized quantized average consensus algorithm with finite-time convergence) to estimate averages of primal variables and gradients, then recovers a BFGS Hessian approximation. 
Both variants are fully decentralized, support quantized communication, and achieve local convergence to a neighborhood of a local optimum for non-convex problems (Theorem~\ref{them: local convergence error1} for Algorithm~\ref{alg:Bilevel Consensus ALADIN} and Appendix~\ref{APP: local convergence error 2} for Algorithm~\ref{alg: Decentralized Reduced BFGS Consensus ALADIN}).

 \textbf{Paper Organization.}
The remainder of the paper is organized as follows. 
Section~\ref{sec: Notation} introduces the notation and essential background. Section~\ref{sec: problem} formulates the problem. Section~\ref{sec: C-ALADIN} presents the C-ALADIN framework. 
Section~\ref{sec: centralized C-ALADIN} describes the centralized C-ALADIN and its convergence analysis for convex and non-convex problems. 
Section~\ref{sec: decentralized algorithm} introduces the quantized decentralized C-ALADIN and its convergence theory. 
Section~\ref{sec: numerical} provides numerical examples and performance comparisons with existing methods. Finally, Section~\ref{sec: conclusion} concludes the paper.

\section{Notation and Preliminaries}\label{sec: Notation}

\textbf{Notation.} 
We use the symbols $\mathbb{R}$, $\mathbb{Q}$, $\mathbb{Z}$, and $\mathbb{N}$ to represent the sets of real, rational, integer, and natural numbers, respectively. 
Matrices are indicated by capital letters (e.g., $A$), and vectors are represented by lowercase letters (e.g., $a$).
The transpose of matrix $A \in \mathbb R^{n \times n}$ and vector $a \in \mathbb R^{n}$ are represented as $A^\top$ and $a^\top$, respectively. 
For a real number $a \in \mathbb R $, $\lfloor a \rfloor$ and $\lceil a \rceil$ denote the greatest integer less than or equal to $a$ and the least integer greater than or equal to $a$, respectively. 
For the real vector $a \in \mathbb R^n$, $\lfloor a \rfloor \in \mathbb R^n$ and $\lceil a \rceil\in \mathbb R^n$ denote the element-wise operation. 
Furthermore, $\mathbf{1}$ represents the vector of all ones and $I$ denotes the identity matrix with appropriate dimensions. 
In addition, $\|a\|$ denotes the Euclidean norm of the vector $a$. 
The value of a variable $a$ of agent $i$ at iteration $k$ is denoted by $a_i^{[k]}$.
Furthermore, $|\mathcal S|$ denotes the cardinality of a countable set $\mathcal S$.

\textbf{Graph Theory.} 
The communication network is captured by a directed graph (or digraph) $\mathcal G=(\mathcal V, \mathcal E)$. 
The set of agents  is denoted as $\mathcal V= \{1, \dots, N\}$ (where $\mathcal N\geq 2$). 
The set of edges is denoted as $\mathcal E \subseteq \mathcal V\times \mathcal V \cup\{(i,i) \ | \ i \in \mathcal V \}$ (each agent has a virtual self-edge).  
A directed edge from agent $i$ to agent $j$ is denoted by $e_{ji} \overset{\cdot}{=}(j,i)\in \mathcal E$. 
The subset of agents that can directly transmit information to agent $i$ is called the set of in-neighbors of $i$ and is denoted $\mathcal N_i^- = \{ j\in \mathcal V \ | \ (i,j)\in \mathcal E \}$. 
The subset of agents that can directly receive information from agent $i$ is called the set of out-neighbors of $i$ and is denoted $\mathcal N_i^+ = \{ l \in \mathcal V \ | \ (l,i) \in \mathcal E \}$. 
The cardinality of $\mathcal N_i^-$ represented as $\mathcal D_i^- = |\mathcal N_i^-|$, is called \emph{in-degree} of agent $i$. 
The cardinality of $\mathcal N_i^+$ represented as $\mathcal D_i^+ = |\mathcal N_i^+|$, is called \emph{out-degree} of agent $i$. 
The diameter $D$ of digraph $\mathcal G$ is the longest shortest path between $i , j \in \mathcal V$.  
A digraph $\mathcal G$ is strongly connected if there exists a directed path from every agent $i$ to agent $j$ that $i,j\in \mathcal V$.

\textbf{Quantization.} 
Quantization reduces bandwidth requirements and improves communication efficiency. Three primary types of quantizers have been extensively studied: asymmetric, uniform, and logarithmic (see \cite{2019:Wei_Johansson}). In this paper, we utilize asymmetric mid‑rise quantizers with an infinite range, although our findings also hold for other types of quantizers. An asymmetric mid‑rise quantizer is defined as
\begin{equation}\label{eq: quantizer} \small
\begin{aligned}
q_{\Delta}(b) = \Delta \left\lfloor \frac{b}{\Delta} \right\rfloor \in \mathbb R^n,
\end{aligned}
\end{equation}
where $b\in \mathbb R^n$ is the value to be quantized, and $\Delta \in \mathbb Q$ denotes the quantization level. 

\section{Problem Formulation}\label{sec: problem}


Consider a network of $N$ agents, where each agent $i$ possesses a local objective function $f_i: \mathbb{R}^n \rightarrow \mathbb{R}$ and let a centralized coordinator variable $z\in \mathbb{R}^n$ be available. The consensus optimization problem is formulated as follows:
\begin{equation}\label{eq: reformulate1}\small
    \begin{aligned}
    \textbf{P1:} \quad \min_{\substack{x_i \in \mathbb{R}^n\\ i = 1, \dots, N}} & \qquad \sum_{i=1}^{N} f_i(x_i) \quad 
    \text{s.t.}  \quad x_i = z. 
    \end{aligned}
\end{equation}
This formulation follows the standard consensus approach \cite[Chapter 7.1]{boyd2011distributed}, where the equality constraints enforce agreement among all local variables $x_i$ on a common global variable $z$. 

Let us now consider a communication network represented by a directed graph $\mathcal{G} = (\mathcal{V}, \mathcal{E})$. Each agent $i$ has exclusive knowledge of its local cost function $f_i: \mathbb{R}^n \rightarrow \mathbb{R}$, and communication links have limited bandwidth. To accommodate fully decentralized operation and quantized communication, we reformulate \eqref{eq: reformulate1} as follows:
\begin{equation}\label{eq: reformulate3}\small
    \begin{aligned}
        \textbf{P2:} \quad \min_{\substack{x_i \in \mathbb{R}^n \\ i = 1, \dots, N}} & \quad \sum_{i=1}^{N} f_i(x_i) \\
        \text{s.t.} \quad & \quad x_i = x_j = z, \quad \forall (i,j) \in \mathcal{E}, \\
        & \quad m_{ij} \in \mathcal{Q}_\Delta, \quad \forall (i,j) \in \mathcal{E},
    \end{aligned}
\end{equation}
where $m_{ij}$ denotes the quantized message transmitted from agent $i$ to agent $j$, and the set $\mathcal{Q}_\Delta$ is defined as $\mathcal{Q}_\Delta =\{ q_\Delta(v) \mid v \in \mathbb{R}^n \}$, with $q_\Delta(\cdot)$ being the asymmetric mid‑rise quantizer given in \eqref{eq: quantizer}. Hence, any transmission of a vector $v \in \mathbb{R}^n$ from agent $i$ to agent $j$ is subject to quantization, i.e., $m_{ij} = q_\Delta(v)$. In contrast to \eqref{eq: reformulate1}, problem \eqref{eq: reformulate3} describes a fully decentralized scenario without a central coordinator. Here, agents coordinate exclusively by exchanging quantized messages with their immediate neighbors.

\subsection*{Operational Assumption}
We first make the following assumption that are important for our subsequent development. 

\begin{assumption}[Network Connection]\label{ass:1} 
    The communication network is modeled as a \textit{strongly connected} digraph $\mathcal{G} = (\mathcal{V}, \mathcal{E})$. 
    Also, every agent $i$ knows the diameter of the network $D$, and a common quantization level $\Delta$. 
\end{assumption}

Assumption~\ref{ass:1} ensures that information can propagate between all agents in the network and guarantees the convergence of our proposed algorithms (since strongly connected digraph implies there is a path from every agent to every other agent in the network).  
Knowing the diameter of the network is useful for each agent to determine whether Algorithm \ref{alg:QuAS} has converged, allowing it to proceed to step~$3$ of \eqref{alg:ALADIN}. 
Finally, the quantization level is important for quantizing the messages as described in \eqref{alg:QuAS} (thus ensuring efficient communication during the execution of Algorithm \ref{alg:QuAS}). 
Note here that any additional assumptions required will be stated explicitly before presenting each result.

\section{From Typical ALADIN to  C-ALADIN}\label{sec: C-ALADIN} 

Distributed resource problems are generally formulated in the fashion of mathematical programming, where $N$
	separable objectives are linearly coupled by $m$ equality constraints. In distributed resource problems, 
	$f_i: \mathbb R^{n_i}\rightarrow \mathbb R$ denotes the closed proper local objective function of each agent $i$. 
	Formally, 
	distributed resource problems can be described as follows \cite{Houska2021}:
	\begin{equation}\label{eq: DOPT_G}\small
		\begin{aligned}
			\min_{\substack{x_i \in \mathbb R^{n_i}\\ i = 1, \dots, N}} 
            \quad \mathop{\sum}_{i=1}^{N}  f_i(x_i)\quad
			\mathrm{s.t.}\;\; \quad\mathop{\sum}_{i=1}^{N} A_ix_i=b. 
		\end{aligned}
	\end{equation}
	Here, the coupling matrices $A_i\in \mathbb R^{n_i\times m}$ and the coupling parameter $b\in \mathbb R^{m}$ are given. The dimensions $n_i$s of local variables $x_i$s are potentially different. 
The augmented Lagrangian of problem \eqref{eq: DOPT_G} is given by $\mathcal L_\rho (x, \Lambda) = \sum^N_{i=1} f_i(x_i) + \Lambda^\top \left({\sum}_{i=1}^{N} A_ix_i-b  \right)
   + \frac{\rho}{2} \left\| \sum_{i=1}^{N} A_ix_i-b \right\|^2,$
 where, $\rho>0$ is a penalty parameter, $\Lambda$ denotes the dual variables and $x=[x_1^\top, \cdots,x_N^\top]^\top$ collects the local primal variables. Focusing on the augmented Lagrangian, details of typical ALADIN  to solve the resource allocation problem in \eqref{eq: DOPT_G} shows as \eqref{eq: T-ALADIN},
	\begin{equation}\label{eq: T-ALADIN}\small
			\left\{
			\begin{aligned}
				&x_i^{[k+1]} = \mathop{\mathrm{\arg}\mathrm{\min}}_{x_i}
				f_i(x_i) + \left(\Lambda^{[k]}\right)^\top A_i x_i + \frac{\rho}{2} \left\| x_i - y_i^{[k]} \right\|^2;\\
                & g_i = \nabla f_i\left(x_i^{[k+1]}\right), B_i^{[k+1]} \approx \nabla^2 f_i\left(x_i^{[k+1]}\right)\succ 0;\\
				& \quad\left(\Delta y,\Lambda^{[k+1]}\right)\\
				=&\left\{\begin{aligned}
					\mathop{\mathrm{\arg}\mathrm{\min}}_{\Delta y_i\in \mathbb{R}^{n_i},\forall i}\; & \sum_{i=1}^N  \left( \frac{1}{2}\Delta y_i^\top B_i^{[k+1]} \Delta y_i + g_i^\top \Delta y_i\right) \\
					\operatorname{s.t.}\quad&  \sum_{i=1}^N A_i\left(x_i^{[k+1]}+\Delta y_i \right) = b\; | \;\Lambda
				\end{aligned}\right\};\\
				&y_i^{[k+1]}=x_i^{[k+1]}+\Delta y_i.
			\end{aligned}\right.
		\end{equation}
        At each iteration, the agents solve their respective augmented Lagrangian subproblems in parallel, yielding the local primal updates $x_i^{[k+1]}$. Subsequently, a coupled QP coordinates the agents' decisions, simultaneously producing the dual variables $\Lambda^{[k+1]}$ and the primal increments $\Delta y$.
        
To solve consensus optimization problems using typical ALADIN \eqref{eq: T-ALADIN}, Problem \eqref{eq: reformulate1} can be reformulated as,
\begin{equation}\small
		\label{eq: DOPT_G2}
		\begin{aligned}
			\min_{\substack{x_i \in \mathbb{R}^n \\ i = 1, \dots, N}}& \quad   f_1(x_1)+f_2(x_2)+\cdots+f_N(x_N)\\
			\mathrm{s.t.}\;\quad\;&  \quad\begin{matrix}
				\;\;\; x_1  &-x_2 & & & &=0 \\
				&\;\;\; x_2 &-x_3 & & &=0\\
				&   \vdots & \vdots& &\vdots & \;\;\;\;\vdots\\
				&    &  & x_{N-1}
				& -x_N&=0\\
				-x_1 &    &  & 
				& \;\;\;x_N&\;=0.\\
			\end{matrix}
		\end{aligned}
	\end{equation}
	The linear constraints in \eqref{eq: DOPT_G2} can be compactly expressed as $\sum_{i=1}^{N} A_i x_i = 0$, where $A_1 = \left(I,
	0,
	\cdots,
	0,
	-I \right)^\top$, $A_2=\left(-I,
	I,
	0,
	\cdots,
	0 \right)^\top$,$\cdots$, $A_N=\left(0,
	\cdots,0,
	-I,
	I \right)^\top$,
	and $I$ denotes the appropriately sized identity matrix. 
However, this formulation has notable limitations. The choice of ${A_i}$ is not unique, and the resulting coupled matrices are often difficult to implement in practice and may contain redundant information. Moreover, the associated communication and computational overhead further restrict the practical applicability of ALADIN, leading to inherent inefficiencies in implementing Problem \eqref{eq: DOPT_G2}.  Therefore, instead of pursuing this line, we directly analyze the consensus formulation \eqref{eq: reformulate1} in the sequel.

The augmented Lagrangian of problem \eqref{eq: reformulate1} is given by
\begin{equation}\label{eq: Lagrangian}\small
   \mathcal L_\rho (x, z, \lambda) = \sum^N_{i=1} f_i(x_i) + \lambda_i^\top \left(x_i -z  \right) + \frac{\rho}{2} \left\| x_i -z \right\|^2,
 \end{equation}
where 
$\lambda=[\lambda_1^\top, \cdots,\lambda_N^\top]^\top$ denotes the dual variables. 
Focusing on \eqref{eq: Lagrangian}, drawing inspiration from Consensus ADMM \cite{boyd2011distributed}, details of C-ALADIN to solve the consensus optimization problem in \eqref{eq: reformulate1} are as follows, 
\begin{equation}\label{alg:ALADIN}\small
	\left\{
	\begin{aligned}
		&	x_i^{[k+1]} = \mathop{\arg\min}_{x_i} f_i(x_i) +\left(\lambda_{i}^{[k]}\right)^\top  x_i + \frac{\rho}{2} \left\|x_i-z^{[k]}\right\|^2;\\
         & g_i = \nabla f_i\left(x_i^{[k+1]}\right), B_i^{[k+1]} \approx \nabla^2 f_i\left(x_i^{[k+1]}\right)\succ 0;\\
		& \quad\left(z^{[k+1]},\Delta x,\lambda^{[k+1]}\right)\\
		=&\left\{\begin{aligned}
			\mathop{\mathrm{\arg}\mathrm{\min}}_{\Delta x_i\in \mathbb{R}^{n},\forall i}\; & \sum_{i=1}^N  \left( \frac{1}{2}\Delta x_i^\top B_i^{[k+1]} \Delta x_i + g_i^\top \Delta x_i\right) \\
			\operatorname{s.t.}\quad&  x_i^{[k+1]}+\Delta x_i  = z\; | \;\lambda_i
		\end{aligned}\right\}.
	\end{aligned}\right.
\end{equation}
C-ALADIN enables each agent to independently solve its nonlinear programming (NLP) subproblem while coordinating with others through a consensus QP. In \eqref{alg:ALADIN}, the first phase consists of each agent updating its local minimizer to $x_i^{[k+1]}$ via an augmented Lagrangian subproblem. Given $x_i^{[k+1]}$, each agent then approximates the positive definite Hessian approximation matrices $B_i$ and evaluates the gradients $g_i$ of $f_i$. In the second phase, upon receiving $B_i$ and $g_i$ from all agents, the master solves a large‑scale convex consensus QP. The solution to this consensus QP yields the updated primal variables $z$, $x_i^{[k+1]}+\Delta x_i$ for all $i$, and the dual variables $\lambda_i$. Finally, the master communicates the updated primal and dual variables back to each agent, and the entire process is repeated until convergence. This use of a consensus QP allows ALADIN to directly address distributed consensus optimization without requiring the manual construction of coupling matrices as in \eqref{eq: DOPT_G2}. The following sections present concrete instantiations of C-ALADIN, where the centralized or decentralized nature of the algorithm stems from whether the consensus QP in \eqref{alg:ALADIN} is solved in a centralized or decentralized manner.

\section{Centralized C-ALADIN}\label{sec: centralized C-ALADIN}
In this section, we present two centralized implementations of C-ALADIN that solve \eqref{eq: reformulate1} by handling the consensus QP in \eqref{alg:ALADIN} in a centralized manner, differing only in the choice of $B_i$. Section~\ref{sec: BFGS Consensus ALADIN} proposes Algorithm~\ref{alg:BFGS ALADIN2}, which uses BFGS updates and a closed-form solution of the consensus QP to reduce computational and communication complexity. When second-order information is not applicable, Section~\ref{sec: Reduced Consensus ALADIN} provides a first-order alternative. Finally, Section~\ref{sec: Convengence of Consensus ALADIN} gives the convergence analysis of centralized C-ALADIN.
\subsection{Second-Order C-ALADIN}\label{sec: BFGS Consensus ALADIN}
	In \eqref{alg:ALADIN}, each agent is required to upload a substantial amount of information to the coordinator. Moreover, the coordinator must solve a large-scale consensus QP. These two factors significantly hinder the computational efficiency of \eqref{alg:ALADIN}, particularly when applied to high-dimensional problems. This section presents two techniques to improve the communication and computation efficiency of C-ALADIN.
    
\subsubsection{Improvement of Uploading Communication Efficiency}\label{sec: Derivative-Free Consensus ALADIN}
					In \eqref{alg:ALADIN}, in order to avoid uploading the gradient and Hessian approximation directly, we choose to decode the first- and second-order information on the master side, which is detailed as follows
					\begin{equation}\small \label{eq: BFGS}
						\left\{
						\begin{aligned}
							&g_i(x_i^{[k+1]})=\;\rho(z^{[k]}-x_i^{[k+1]})-\lambda_i^{[k]}\quad\text{((sub)gradient)},\\
							&s_i(x_i^{[k+1]},x_i^{[k]})=\;x_i^{[k+1]}-x_i^{[k]},\\
							&y_i(x_i^{[k+1]},x_i^{[k]})=\; g_i(x_i^{[k+1]})-g_i(x_i^{[k]}),\\
							&B_i^{[k+1]}=\;B_i^{[k]}-\frac{B_i^{[k]} s_i s_i^\top B_i^{[k]}}{s_i^\top B_i^{[k]} s_i}+\frac{y_iy_i^\top}{s_i^\top y_i}.
						\end{aligned}
						\right.
					\end{equation}
    First, assuming that the augmented NLP in \eqref{alg:ALADIN} can be solved exactly, the Clarke subdifferential of $f_i$ at $x_i^{[k+1]}$ yields $g_i = \rho(z^{[k]} - x_i^{[k+1]}) - \lambda_i^{[k]} \in \partial f_i$. Consequently, the (sub)gradient can be recovered from $x_i^{[k+1]}$ without explicit transmission. Second, by leveraging the differences of local private variables and corresponding (sub)gradients, the BFGS Hessian approximation can also be reconstructed at the coordinator side. A detailed analysis of the reduction in agent-to-coordinator data uploads achieved by the proposed approach can be found in \cite[Remark 1]{Du2025ACC}.
\subsubsection{Improvement of Download Communication and Computation Efficiency}\label{sec: download}
The consensus QP in  \eqref{alg:ALADIN} admits the following closed-form solution:
\begin{subequations}\label{eq: closed form}
\begin{align}
&z^{[k+1]}=\left(\sum_{i=1}^N B_i^{[k+1]} \right)^{-1} \sum_{i=1}^N \left(B_i^{[k+1]}x_i^{[k+1]} - g_i \right), \label{eq: closed form1}\\
&\lambda_{i}^{[k+1]}=B_i^{[k+1]}(x_i^{[k+1]}-z^{[k+1]})-g_i.\label{eq: closed form2}
\end{align}
\end{subequations}
A similar update formula for $z$ can also be found in \cite{mokhtari2018iqn,kovalev2019stochastic,dinc2024incremental}. The resulting reduction in computational complexity achieved by this closed-form solution is discussed in \cite[Remark 2]{Du2025ACC}.


\subsubsection{Algorithm Development}

By combining the techniques from the two preceding subsections, we present the second-order variant of C-ALADIN in Algorithm~\ref{alg:BFGS ALADIN2}.

\begin{algorithm}[ht]\small
\caption{Second-Order C-ALADIN}
\textbf{Initialization:} choose $\rho>0$, initial guess $(\lambda_i^{[1]},z^{[1]},B_i=\rho I)$.
\textbf{Iteration.}
\begin{enumerate}
    \item Each agent updates its local variable $x_i$ locally and transmits it to the master:
    \begin{equation}\label{eq: NLP}\small
        {x_i}^{[k+1]}=\mathop{\arg\min}_{x_i}   f_i(x_i) +\left(\lambda_{i}^{[k]}\right)^\top  x_i + \frac{\rho}{2} \left\|x_i-z^{[k]}\right\|^2.
    \end{equation}
    \item Operations on the master side:
    \begin{enumerate}
        \item Recover the gradient $g_i$ and the BFGS Hessian $B_i^{[k+1]}$ from each ${x_i}^{[k+1]}$ using \eqref{eq: BFGS}.
        \item Update the global variable $z^{[k+1]}$ and the dual variables $\lambda_i^{[k+1]}$ via \eqref{eq: closed form}.
    \end{enumerate}
\end{enumerate}
\textbf{Output.} Each agent $i$ obtains $x_i^*$ that solves problem \eqref{eq: reformulate1}.
\label{alg:BFGS ALADIN2}
\end{algorithm}

Following the same structure as \eqref{alg:ALADIN}, each agent first updates its local variable $x_i^{[k+1]}$ and sends it to the master. In the second step, the master reconstructs the gradients and Hessian approximations of all agents using the BFGS update (see \eqref{eq: BFGS}). In the third step, the master computes the global variable $z^{[k+1]}$ and the dual variables $\lambda_i^{[k+1]}$ via \eqref{eq: closed form} and broadcasts them to the agents. The entire procedure is repeated until convergence of $z$.
	

\subsection{First-Order C-ALADIN}\label{sec: Reduced Consensus ALADIN}

To further improve the computational efficiency of \eqref{eq: closed form}, we next introduce a first-order variant of C-ALADIN. This variant differs from the second-order C-ALADIN in that $B_i = \rho I$. Consequently, the updates of $x_i$, $g_i$, $z$, and $\lambda_i$ are given by the following equations:
\begin{subequations}\label{eq: RCA}\small
\begin{align}
       {x_i}^{[k+1]} &= \mathop{\arg\min}_{x_i}   f_i(x_i) +\left(\lambda_{i}^{[k]}\right)^\top  x_i + \frac{\rho}{2} \left\|x_i-z^{[k]}\right\|^2, \label{eq: local update} \\
        g_i &= \rho\left(z^{[k]}-x_i^{[k+1]}\right)-\lambda_i^{[k]}, \quad \forall i \in \mathcal V , \label{eq: gradient1} \\
        z^{[k+1]} &= \frac{1}{N}\sum_{i=1}^N\left(x_i^{[k+1]}-\frac{g_i}{\rho}\right) , \label{eq: z update} \\
        \lambda_i^{[k+1]} &= \rho\left(x_i^{[k+1]}-z^{[k+1]}\right)-g_i, \quad \forall i \in \mathcal V \label{eq: dual update}.  
\end{align}
\end{subequations}
Note that both Algorithm~\ref{alg:BFGS ALADIN2} and \eqref{eq: RCA} are two variants of C-ALADIN (see \eqref{alg:ALADIN}). The difference between \eqref{eq: RCA} and Consensus ADMM is discussed in \cite[Appendix]{Du2025ACC}.

\subsection{Convergence Analysis}\label{sec: Convengence of Consensus ALADIN}

In this section, we present the convergence analysis for centralized C-ALADIN. Specifically, Section~\ref{sec: local} provides the local convergence analysis for non-convex consensus optimization problems, while Section~\ref{sec: global} develops the global convergence theory for convex consensus optimization problems. These convergence results apply directly to the two variants of C-ALADIN, namely Algorithm~\ref{alg:BFGS ALADIN2} and \eqref{eq: RCA}, as well.

\subsubsection{Local Convergence Analysis}\label{sec: local}
The local convergence analysis of C-ALADIN for non-convex problems follows an approach similar to \cite[Section 7]{Houska2016}. First, we introduce the following assumption to ensure that stationary points correspond to strict local minimum for problem \eqref{eq: reformulate1}.

\begin{assumption}[Second Order Sufficient Condition]\label{ass:02} 
    The local cost function $f_i$ of each agent $i\in \mathcal V$ is closed, proper, twice continuously differentiable and let Second Order Sufficient Condition (SOSC) \cite[Theorem 5.4]{diehl2016lecture}, \cite{sun2006optimization} been satisfied, such that
    \begin{equation}\label{eq: SOSC}
    \nabla^2f_i(z^*)\succ 0,\; \forall i \in \mathcal V,
    \end{equation}
where $z^*$ denotes a strict local minimizer of problem \eqref{eq: reformulate1}.
\end{assumption}
This condition supports local convergence analysis of C-ALADIN for non-convex problems, as it does not require global convexity of the local cost functions $f_i$ in problem \eqref{eq: reformulate1}.

	\begin{theorem}[Local Convergence of Algorithm \ref{alg:BFGS ALADIN2}]\label{them: local convergence}
Consider problem \eqref{eq: reformulate1} with variables $x_i$ updated according to \eqref{eq: NLP}. Under Assumption \ref{ass:02}, satisfied at a local minimizer $z^*$ of \eqref{eq: reformulate1}, if the variables $x_i$ and $z$ are initialized in a sufficiently small neighborhood of $z^*$, then the sequence generated by C-ALADIN converges locally to $z^*$.
	\end{theorem}
	\textbf{Proof:} See  \cite[Theorem 3]{Du2025ACC}. \hfill$\blacksquare$ 

In the above section, we show the  local convergence of C-ALADIN. In the next section we will show global convergence analysis  of C-ALADIN.

\subsubsection{Global Convergence Analysis}\label{sec: global}

	We assume that the  $f_i$s of \eqref{eq: reformulate1} are closed, proper, and strictly convex. 
	In the global convergence theory of C-ALADIN, the sub-problems are updated as ${x_i}^{[k+1]}=\mathop{\arg\min}_{x_i} f_i(x_i)+ \left(\lambda_i^{[k]} \right)^\top  x_i+\frac{1}{2}\left\|x_i-z^{[k]}\right\|_{B_i}^2$, instead of \eqref{eq: NLP}.

		For ease of analysis, we temporarily assume that the matrices $B_i \succ 0$ are constant. In this setting, Algorithm \ref{alg:BFGS ALADIN2} can be expressed as follows:
\begin{subequations}\label{eq: BFGS ALADIN}\small
	 \begin{align}
			&{x_i}^{[k+1]}=\mathop{\arg\min}_{x_i}f_i(x_i)+\left(\lambda_i^{[k]}\right)^\top  x_i+\frac{1}{2}\left\|x_i-z^{[k]}\right\|_{B_i}^2,\label{eq: x update2}\\
			&g_i=B_i\left(z^{[k]}-x_i^{[k+1]}\right)-\lambda_i^{[k]},\ \forall i \in \mathcal V  \label{eq: gradient}\\
			&z^{[k+1]}=\left(\sum_{i=1}^N  B_i \right)^{-1}  \sum_{i=1}^N \left(B_ix_i^{[k+1]} -    g_i   \right), \label{eq: global z}\\
			&\lambda_{i}^{[k+1]}=B_i\left(x_i^{[k+1]}-z^{[k+1]}\right)-g_i, \ \forall i \in \mathcal V.\label{eq: dual}
		\end{align}
	\end{subequations}
Note that the following relation holds and will be used in the convergence proof:
\begin{equation}\label{eq: sum_lam}\small
    \sum_{i=1}^{N} \lambda_i^{[k]} = 0, \;\forall k.
\end{equation}

	For establishing the global convergence theory of  C-ALADIN,		
	we introduce the following \emph{energy function} \cite{ling2015dlm} (also called Lyapunov function in \cite[Appendix A]{boyd2011distributed} and \cite{yang2022proximal}),
	\begin{equation}\label{eq: LYA}\small
		\mathscr{L}(z,\lambda)=	 \sum_{i=1}^{N}\left(   \left\|z-z^*\right\|_{B_i}^2+\left\|\lambda_i-\lambda_i^*\right \|_{B_i^{-1}}^2 \right),
	\end{equation}
    where  $(z^*, \lambda^*)$ denotes the optimal solution pair of problem \eqref{eq: reformulate3}.
	Note that the choice of energy function is not unique \cite{nonlinear}. Next, we will establish the global convergence of C-ALADIN by demonstrating the monotonic decrease of the energy function \eqref{eq: LYA}.

Before providing the global convergence theory of C-ALADIN, we first proposed the following lemma. 

\begin{lemma}
For the distributed consensus optimization problem presented in \eqref{eq: reformulate1}, \eqref{alg:ALADIN} establishes a relationship between the local primal update $x_i^{[k+1]}$, the local dual variables $\lambda_i$ and $\lambda_i^{[k+1]}$, and the global primal variable approximations $ z_i$ and $\ z_i^{[k+1]}$. 
This relationship is 
    \begin{equation}\label{eq: lemma}\small
        x_i^{[k+1]} = \frac{B_i^{-1}\left( \lambda_i^{[k+1]}-\lambda_i^{[k]}\right)}{2}+\frac{ z^{[k+1]}+ z^{[k]}}{2}.
    \end{equation}
\end{lemma}
\textbf{Proof:} From \eqref{eq: dual} we have 
\begin{equation}\label{eq: lemma2}\small
    \begin{aligned}
    \lambda_i^{[k+1]}&=B_i\left(x_i^{[k+1]}-z^{[k+1]}\right)-g_i\\
        &\overset{\eqref{eq: gradient}}{=} B_i\left(x_i^{[k+1]}- z^{[k+1]}\right) -  \left(B_i\left( z^{[k]}-x_i^{[k+1]}\right)-\lambda_i^{[k]} \right)\\
        &= 2B_i x_i^{[k+1]} - B_i(z^{[k+1]}+ z^{[k]}) +\lambda_i^{[k]}.
    \end{aligned}
\end{equation}
equation \eqref{eq: lemma} is then derived from \eqref{eq: lemma2}. \hfill$\blacksquare$

Note that when $B_i = \rho I$, \eqref{eq: lemma} becomes
\begin{equation}\label{eq: lemma3}\small
   x_i^{[k+1]} = \frac{\lambda_i^{[k+1]} - \lambda_i^{[k]}}{2\rho} + \frac{z^{[k+1]} + z^{[k]}}{2},
\end{equation}
which will be employed in the convergence analysis of \eqref{eq: RCA}.

Before presenting the global convergence analysis, we introduce the following assumption. The strict convexity of $f_i$, together with the existence of an optimal solution, guarantees the uniqueness of the global optimal solution for problem \eqref{eq: reformulate1}.
\begin{assumption}[Strictly Convex]\label{ass:01} 
    The local cost function $f_i$ of each agent $i\in \mathcal{V}$ is closed, proper, and strictly convex, satisfying
    \begin{equation}\label{eq: convex2}\small
        f_i(t x_\alpha + (1-t) x_\beta) < t f_i(x_\alpha) + (1-t) f_i(x_\beta),
    \end{equation}
    for every $x_\alpha, x_\beta \in \mathbb{R}^n$ and $t \in (0,1)$.
\end{assumption}


\begin{theorem}[Global Convergence of \eqref{eq: BFGS ALADIN}]\label{The: 1}
Let Assumption~\ref{ass:01} hold for Problem~\eqref{eq: reformulate1}.  
Assume that Problem~\eqref{eq: reformulate1} is feasible and satisfies strong duality, and that, for each \(i\), the matrices \(B_i \succ 0\) are constant and symmetric. Let $x_i^* = z^*$ denotes the (unique) primal solution and $\lambda^*$ a dual solution of problem~\eqref{eq: reformulate1}, then the iterates of \eqref{eq: BFGS ALADIN} satisfy
		\begin{equation}\small
			\mathscr{L}\left(z^{[k+1]},\lambda^{[k+1]}\right) \leq \mathscr{L}\left(z^{[k]},\lambda^{[k]}\right) - 4\Pi\left( \sum^N_{i=1} \|x_i^{[k+1]} - z^*\| \right),
		\end{equation}
		for a class $\mathcal{K}$ function $\Pi$ \cite{nonlinear}.		
	\end{theorem}
	\textbf{Proof:} See  \cite[Theorem 1]{Du2025ACC}.
    \hfill$\blacksquare$

It should be noted that the aforementioned theorem proves the monotonic decrease of \eqref{eq: LYA} throughout the iterations of C-ALADIN. However, due to the lack of an explicit analytical form for $\Pi(\cdot)$, the analysis guarantees sublinear convergence of C-ALADIN under Assumption~\ref{ass:01}. In the case of smooth and strongly convex functions $f_i$, a stronger result can be established, namely the global linear convergence of the algorithm. Prior to presenting this result, we introduce the following assumption.

\ADD{
\begin{assumption}[Strong Convexity and Smoothness]\label{ass:2} 
The local cost function $f_i$ of each agent $i \in \mathcal{V}$ is closed, proper, $L$-smooth and $\mu$-strongly convex. 
Specifically, for each local cost function $f_i$, for every $x_\alpha,x_\beta \in \mathbb R^n$, there exists a \emph{strong-convexity} constant $\mu_i>0$ and a \emph{Lipschitz-continuity} constant $L_i>0$ such that 
\begin{equation}\label{eq: mu}\small
\begin{aligned}
    f_i(x_\alpha)+&\nabla f_i(x_\alpha)^\top (x_\beta-x_\alpha) +\frac{\mu_i}{2}\|x_\beta-x_\alpha\|^2\leq f_i(x_\beta),
\end{aligned}
\end{equation}
and
\begin{equation}\label{eq: lip}\small
\begin{aligned}  
     \left\|\nabla f_i(x_\alpha)-\nabla f_i(x_\beta) \right\| \leq L_i \left\|x_\alpha-x_\beta\right\|,
\end{aligned}
\end{equation}
hold.
\end{assumption}
Assumption~\ref{ass:2} ensures strong convexity and $L$-smoothness. This allows us to establish a global linear convergence rate for \eqref{alg:ALADIN} while guaranteeing that the global cost function in \eqref{eq: reformulate3} has a unique minimum, a standard requirement in first-order distributed optimization frameworks \cite{beck2017first}.}
\begin{theorem}[Global Linear Convergence of \eqref{eq: BFGS ALADIN}]\label{theorem 2} 
Consider Problem~\eqref{eq: reformulate1} under Assumption~\ref{ass:2} (Strong Convexity and Smoothness).  
Assume that Problem~\eqref{eq: reformulate1} is feasible and satisfies strong duality, and that, for each \(i\), the matrices \(B_i \succ 0\) are constant.  Let $(z^*,\lambda^*)$ denote the unique optimal primal-dual solution.
Then, for the iterates of \eqref{eq: BFGS ALADIN}, there exists a constant $\delta > 0$ such that the following key inequality holds:
\begin{equation}\label{eq: dela}\small
    \delta \mathscr{L}\left(z^{[k+1]},\lambda^{[k+1]}\right) \leq 4 \sum_{i=1}^{N} \mu_i\left\| x_i^{[k+1]} - z^*\right \|^2.
\end{equation}
This implies that the Lyapunov function decreases linearly at each iteration:
\begin{equation}\label{eq: descent}\small
    \mathscr{L}\left(z^{[k+1]},\lambda^{[k+1]}\right) \leq \frac{1}{1+\delta} \mathscr{L}\left(z^{[k]},\lambda^{[k]}\right),
\end{equation}
and consequently, the algorithm converges globally with a linear rate:
\begin{equation}\label{eq: linear convergence proof1}\small
    \mathscr{L}\left(z^{[k]},\lambda^{[k]}\right) \leq \left( \frac{1}{1+\delta} \right)^k \mathscr{L}\left(z^{[1]},\lambda^{[1]}\right).
\end{equation}
Here, $z^{[1]}$ and $\lambda^{[1]}$ represent the initial primal and dual variables, respectively.
\end{theorem}
    
	\textbf{Proof:} See Appendix \ref{APP: theorem 2}.  \hfill$\blacksquare$


	A tutorial example that satisfies the conditions of Theorem~\ref{theorem 2} can be found in \cite[Section~5.1]{Houska2021}. Moreover, by setting $B_i = \rho I$ and applying \eqref{eq: lemma3}, the global convergence theory of \eqref{eq: RCA} can be established as well. A similar convergence proof for the standard ALADIN (see \eqref{eq: T-ALADIN}) is available in \cite[Chapter~5]{ALADINtheisi}.


Note that the centralized C-ALADIN, including Algorithm~\ref{alg:BFGS ALADIN2} and the variant in \eqref{eq: RCA}, is implemented by agents in a distributed manner but still requires a central coordinator that can exchange messages with every agent in the network. More specifically, the updates in \eqref{eq: closed form1} and \eqref{eq: z update} are performed centrally by the coordinator and rely on the exchange of real-valued vectors. Motivated by this limitation, the next section presents quantized decentralized C-ALADIN. The proposed algorithm enables agents to (i) collaboratively solve problem \eqref{eq: reformulate3} by communicating exclusively with their immediate neighbors, thereby eliminating the need for a central coordinator, and (ii) achieve quantized communication within the network.

\section{Quantized Decentralized C-ALADIN}\label{sec: decentralized algorithm}
This section focuses on the decentralized solution of problem~\eqref{eq: reformulate3} within the C-ALADIN framework, while ensuring high communication efficiency through quantized information exchange. Specifically, in Section~\ref{sec: Decentralized RC-ALADIN}, we introduce a first-order decentralized and quantized variant of C-ALADIN based on~\eqref{eq: RCA}, designed to solve convex instances of problem~\eqref{eq: reformulate3}. Subsequently, in Section~\ref{sec: Decentralized BFGS C-BFGS ALADIN with Efficient Communication}, we propose two second-order decentralized variants derived from Algorithm~\ref{alg:BFGS ALADIN2}, which are capable of addressing non-convex optimization problems. Finally, Section~\ref{sec: convergence2} presents the convergence analysis of all proposed algorithms, establishing global convergence for convex problems and local convergence for non-convex cases.

\subsection{Quantized Decentralized First-Order C-ALADIN}\label{sec: Decentralized RC-ALADIN}

In this section, we present the quantized decentralized first-order C-ALADIN, 
detailed below as \eqref{eq: QuDRC-ALADIN}.

\begin{subequations}\label{eq: QuDRC-ALADIN}\small
	 \begin{align}
			 x_i^{[k+1]} &= \mathop{\arg\min}_{x_i} f_i(x_i)+ \left(\hat\lambda_i^{[k]}\right)^\top x_i + \frac{\rho}{2}\left\|x_i-\hat z_i^{[k]}\right\|^2,\label{eq: new local primal}\\
			g_i &= \rho\left(\hat  z_i^{[k]}-x_i^{[k+1]}\right)-\hat\lambda_i^{[k]},  \label{eq: new gradient}\\
			\hat z_i^{[k+1]}&= \text{Algorithm~\ref{alg:QuAS}} \left(x_i^{[k+1]}-\frac{g_i}{\rho}, D, \Delta \right), \label{eq: local copy update}\\
			 \hat\lambda_i^{[k+1]}&=\rho\left(x_i^{[k+1]}-\hat z_i^{[k+1]}\right)-g_i.\label{eq: new dual}
		\end{align}
	\end{subequations}

The intuition of \eqref{alg:ALADIN} is organized into two main phases: local optimization and coordination among agents. 
In the first step, each agent $i \in \mathcal{V}$ performs a local optimization to determine the optimal value of its variable $x_i$ by solving its corresponding augmented objective function, (see \eqref{eq: new local primal}). 
Following this, in the second step each agent $i$ evaluates the (sub)gradient of its local function $f_i$ at the locally optimized solution $x_i^{[k+1]}$. 
This (sub)gradient evaluation serves as preparation for the aggregation process in the subsequent step (see \eqref{eq: new gradient}). 
In the third step, all agents collaborate to update their estimates of the global variable $\hat z_i^{[k+1]}$ through the quantized, decentralized operation of Algorithm~\ref{alg:QuAS} (see \eqref{eq: local copy update}). 
Finally, in the fourth step each agent updates its dual variable $\hat\lambda_i^{[k+1]}$ which encodes sensitivity information related to the constraints of problem~\eqref{eq: reformulate3}. 
The updated dual variables are then utilized in the next iteration's local optimization phase (see \eqref{eq: new dual}). 
Overall, \eqref{alg:ALADIN} alternates between performing local optimization (step~$1$) and ~\eqref{eq: closed form}, iterating until convergence is achieved and the optimal solution is obtained.
 
\begin{algorithm}[h]
	\small
	\caption{FQAC: Finite-time Quantized Average Consensus}
	\textbf{Input.} 
    $y_i, D, \Delta$.\\
    \textbf{Initialization.} Each agent $i \in \mathcal{V}:$
    \begin{enumerate}
	\item Assigns probability 
    \begin{equation} p_{li}=\left\{
   \begin{aligned}
       &\frac{1}{1+\mathcal D_i^+},\quad& \text{if}\; l\in \mathcal N_i^+\cup \{i\},\\
       &0,\quad & \text{if}\; l\notin \mathcal N_i^+\cup \{i\},
   \end{aligned}
      \right.   
    \end{equation} 
    to each out-neighbor of agent $i$.
    \item Sets $\xi_i = 2$, $\chi_i= 2 q_{\Delta}(y_i)$ (see \eqref{eq: quantizer}).
	\end{enumerate}
	\textbf{Iteration.} For time steps $t=1,2,\cdots$ each agent $i \in \mathcal V$ does: 
	\begin{enumerate}
	\item \textbf{If} $t\;\text{mod}(D) = 1$, sets $M_i = \left\lceil \frac{\chi_i}{\xi_i} \right\rceil$ and $m_i=\left\lfloor \frac{\chi_i}{\xi_i} \right\rfloor$.\vspace{2mm}

\item 
Broadcasts $M_i, m_i$ to each out-neighbor $l \in \mathcal N_i^+$ and receives $M_j, m_j$ from each in-neighbor $j \in \mathcal N_i^-$. 
Then, sets $M_i = \text{max}_{j \in \mathcal N_i^- \cup \{i\}}\; M_j$, $m_i = \text{min}_{j \in \mathcal N_i^-\cup \{i\}}\; m_j$. \vspace{2mm}
\item Sets $\tau_i = \xi_i$. \vspace{2mm}
\item \textbf{While} $\tau_i>1$ \textbf{do}
\begin{enumerate}
    \item $c_i=\left\lfloor \frac{\chi_i}{\xi_i} \right\rfloor$. \vspace{2mm}
    \item Sets $\chi_i=\chi_i - c_i$, $\xi_i=\xi_i-1$, $\tau_i=\tau-1$. 
    \item Transmits $c_i$ to randomly chosen out-neighbor $l \in \mathcal N_i^+\cup \{i\}$ with probability $p_{li}$. 
    \item Receives $c_i$ from $j\in \mathcal N_i^-$ and updates
    \begin{subequations}\label{eq: local information avg}\small
    \begin{align}
    \chi_i^{[t+1]} &= \chi_i^{[t]} + \sum_{j\in \mathcal N_i^- }w_{ij}^{[t]} c_j^{[t]}, \\
    \xi_i^{[t+1]} &= \xi_i^{[t]} + \sum_{j\in \mathcal N_i^-} w_{ij}^{[t]}.
        \end{align}
    \end{subequations}    
    Here $w_{ij}^{[t]} = 1$ if agent $i$ receives $c_j^{[t]}$ from agent $j$ at step $t$. 
    Otherwise $w_{ij}^{[t]}=0$ and agent $i$ does not receive information from agent $j$. 
\end{enumerate}
\item \textbf{if} $t\; \text{mod}\; (D)=0$ and $\left\|M_i-m_i\right\|_{\infty}\leq 1$, set $\hat z_i^{[k+1]} = m_i \Delta$, and stop the operation of the algorithm. 
	\end{enumerate}
    \textbf{Output.} $\hat z_i^{[k+1]}$.
	\label{alg:QuAS}
\end{algorithm}

Algorithm~\ref{alg:QuAS} follows a structure similar to \cite[Algorithm~$1$]{rikos2022non}, and consists of three main operations: quantization, averaging, and a stopping criterion.  
During initialization each agent $i$ quantizes its local information $y_i = x_i^{[k+1]} - \frac{g_i}{\rho}$ into a quantized value $\chi_i$. 
Then, it splits $\chi_i$ into $\xi_i$ pieces (the value of some pieces might be greater than others by one). 
It retains the piece with the smallest value to itself and transmits the rest $\xi_i - 1$ pieces to randomly chosen out-neighbors $l \in \mathcal N_i^+$ or to itself. 
Then, it receives the pieces $c_j$ transmitted from each in-neighbor $j \in \mathcal N_i^-$ and updates $\chi_i$ and $\xi_i$ as in \eqref{eq: local information avg}. 
The algorithm also performs max- and min-consensus operations every $D$ time steps. 
If $\|M_i-m_i\|\leq 1$, then each agent $i$ scales its solution according to the quantization level to compute $\hat z_i^{[k+1]}$.
At this point, Algorithm~\ref{alg:QuAS} terminates, and each agent $i$ transitions to step~$4)$ of \eqref{alg:ALADIN}. 
Note that Algorithm~\ref{alg:QuAS} is guaranteed to converge in finite time as established in \cite[Theorem~$1$]{rikos2022non}. A detailed comparison between the algorithm in \eqref{eq: QuDRC-ALADIN} and related approaches is provided in \cite{Du2025ACC}.

\subsection{Quantized Decentralized Second-Order C-ALADIN}\label{sec: Decentralized BFGS C-BFGS ALADIN with Efficient Communication}

We propose two decentralized second-order variants of C-ALADIN based on Algorithm~\ref{alg:BFGS ALADIN2}, which are designed to handle non-convex problems.

\subsubsection{Bi-level Variant of Second-Order Quantized Decentralized C-ALADIN}\label{sec: Bi-level Decentralized Consensus ALADIN}

In this section, we present a bi‑level variant of decentralized second‑order C‑ALADIN for solving \eqref{eq: reformulate3}. The proposed algorithm is shown in Algorithm~\ref{alg:Bilevel Consensus ALADIN}.

\begin{algorithm}[ht]
		\caption{Quantized decentralized second-order C-ALADIN (Version I)}
        \textbf{Input.} Strongly connected digraph $\mathcal{G} = (\mathcal{V}, \mathcal{E})$, parameter $\rho$, network diameter $D$, quantization level $\Delta$, for each agent $i \in \mathcal{V}$. 
    Each agent $i \in \mathcal V$ has a local cost function $f_i$. 
    Assumptions~\ref{ass:1} and \ref{ass:02} hold. 
    \\
    \textbf{Initialization.} Randomly chosen dual variable $\hat \lambda_i \in \mathbb{R}^n$, and global variable estimation $\hat z_i \in \mathbb{R}^n$, for each agent $i \in \mathcal{V}$. 
    \\
    \textbf{Iteration.} 
	Each agent $i \in \mathcal{V}$ repeats:
        \textbf{Outer level iteration.}
		\begin{enumerate}
			\item Optimize $x_i$ as equation \eqref{eq: new local primal}.
			\item Evaluate and upload the Hessian approximation $B_{i}\succ 0$ and the gradient $g_{i}$ at ${x_i}^{[k+1]}$ as equation  \eqref{eq: BFGS}. The above information are collected for the \emph{inner level} procedure.
            \end{enumerate}
			\textbf{Inner level iteration.}\\
Solve the coupled QP in \eqref{alg:ALADIN} by  \eqref{eq: QuDRC-ALADIN} with the uploaded $x_i^{[k+1]},\;B_i^{[k+1]},\;g_i$ from each agent if
\begin{equation}\label{eq: inner level stop}
    \left\|\begin{split}
        \hat z_i^{[r+\tau]} -\hat z_i^{[r]}
    \end{split}\right\| \leq \epsilon,
\end{equation}
where $r$ and $(r+\tau)$ represent the random selection iterations of the internal hierarchical loop. Output $(\hat z_i^{[r+\tau]}, \hat \lambda_i^{[r+\tau]})$ to the \emph{outer level} procedure.

		\textbf{Output.} Each agent $i$ calculates $x_i^*$ that solves problem \eqref{eq: reformulate3}. 
		\label{alg:Bilevel Consensus ALADIN}
	\end{algorithm}


At the outer level, in the first step, each agent $i \in \mathcal{V}$ updates its local variable $x_i^{[k+1]}$, as in \eqref{eq: new local primal}. In the subsequent step, each agent evaluates its local Hessian and gradient at $x_i^{[k+1]}$ as in \eqref{eq: BFGS}. This prepares the strictly convex consensus QP at the outer level in \eqref{alg:ALADIN}. Then, the inner level solves this consensus QP using any decentralized consensus optimization algorithm; here we take \eqref{eq: QuDRC-ALADIN} as an example.
    
\subsubsection{Decentralized Approximate Second-Order C-ALADIN}\label{sec: Decentralized Approximate BFGS Consensus ALADIN}

Instead of the bi-level structure described above, an alternative approach to preserving second-order information in decentralized C-ALADIN is to employ an inexact aggregation of the Hessian matrices, i.e., $\sum_{i=1}^N B_i^{[k+1]}$ in \eqref{eq: closed form}. The proposed decentralized approximate second-order variant of C-ALADIN is presented in Algorithm~\ref{alg: Decentralized Reduced BFGS Consensus ALADIN}. The difference between Algorithm~\ref{alg: Decentralized Reduced BFGS Consensus ALADIN} and Algorithm~\ref{alg:BFGS ALADIN2} is that the former does not directly use \eqref{eq: closed form} to update $z$ and $\lambda_i$, because it operates in a decentralized, coordinator-free manner. Instead, it first uses \eqref{eq: x AVG} and \eqref{eq: g avg} to compute the averages $\bar x_i^{[k+1]}$ and $l_i^{[k+1]}$ of the primal variables $x_i^{[k+1]}$ and the gradients $g_i$, respectively. It then estimates the Hessian at $\bar x_i^{[k+1]}$ using \eqref{eq: hessian estimation}. Finally, it updates $(\hat z_i^{[k+1]}, \hat \lambda_i^{[k+1]})$ via \eqref{eq: primal_dual1}. This introduces an approximation error compared to Algorithm~\ref{alg:BFGS ALADIN2}. The corresponding convergence analysis is presented in the next section.
\begin{algorithm}[ht]\small
    \caption{Quantized decentralized second-order C-ALADIN (Version II)}
	\textbf{Input.} Strongly connected digraph $\mathcal{G} = (\mathcal{V}, \mathcal{E})$, parameter $\rho$, network diameter $D$, quantization level $\Delta$, for each agent $i \in \mathcal{V}$. 
    Each agent $i \in \mathcal V$ has a local cost function $f_i$. 
    Assumptions~\ref{ass:1} and \ref{ass:02} hold. 
    \\
    \textbf{Initialization.} Randomly chosen dual variable $\hat \lambda_i \in \mathbb{R}^n$, and global variable estimation $\hat z_i \in \mathbb{R}^n$, for each agent $i \in \mathcal{V}$. 
    \\
    \textbf{Iteration.} 
	Each agent $i \in \mathcal{V}$ repeats:
	\begin{enumerate}
		\item Optimize $x_i$ as equation \eqref{eq: new local primal}.
		\item Evaluate Hessian and gradient from ${x_i}^{[k+1]}$:
		\begin{itemize}
			\item Update the gradient $g_i$ as equation \eqref{eq: new gradient}.
			\item Update the BFGS Hessian $B_i^{[k+1]}$ as
            \eqref{eq: BFGS}.

		\end{itemize}

		\item Update the Inexact Hessian aggregation part of \eqref{eq: closed form}:
		\begin{itemize}
			\item Calculate the global variable  averaging estimation $\bar x_i^{[k+1]}$ as 
            \begin{equation}\label{eq: x AVG}
        \bar x_i^{[k+1]}= \text{Algorithm~\ref{alg:QuAS}} \left(x_i^{[k+1]}, D, \Delta \right).
    \end{equation}
           \item Calculate the averaging estimation of gradient averaging at  $l_i^{[k+1]}$ as 
           \begin{equation}\label{eq: g avg}
              l_i^{[k+1]} = \text{Algorithm~\ref{alg:QuAS}} \left(\nabla f_i(\bar x_i^{[k+1]}), D, \Delta \right).
           \end{equation}
				\item Update the non-inversion BFGS Hessian update as
				\begin{equation}\label{eq: hessian estimation}
                \left\{
					\begin{aligned}
						&\bar s(\bar x_i^{[k+1]},\bar x_i^{[k]})=\;\bar x_i^{[k+1]}-\bar x_i^{[k]},\\
						&\bar y(\bar x_i^{[k+1]},\bar x_i^{[k]})=\; l_i^{[k+1]}-l_i^{[k]},\\
						 &H_i^{[k+1]}=\; \left( I - \frac{\bar s \bar y^\top}{\bar y^\top \bar  s}\right)H_i^{[k]}\left( I - \frac{ \bar y \bar  s^\top}{\bar y^\top \bar s}\right)+ \frac{\bar s \bar s^\top}{\bar y^\top \bar s}.
					\end{aligned}
					\right.
				\end{equation}
		\end{itemize} 
	\item Update $( z_i^{[k+1]},\hat \lambda_{i}^{[k+1]})$ as 
	\begin{equation}\label{eq: primal_dual1}
    \left\{
		\begin{aligned}
			\hat z_i^{[k+1]} =&\; \bar x_i^{[k+1]} - H_i^{[k+1]}l_i^{[k+1]},	\\
			\hat \lambda_{i}^{[k+1]}=& \;B_i^{[k+1]}(x_i^{[k+1]}-\hat z^{[k+1]}) -g_i.
		\end{aligned}\right.
	\end{equation}	
	\end{enumerate}
    \textbf{Output.} Each agent $i$ calculates $x_i^*$ that solves problem \eqref{eq: reformulate3}. 
	\label{alg: Decentralized Reduced BFGS Consensus ALADIN}
\end{algorithm}




\subsection{Convergence Analysis}\label{sec: convergence2}
In this section, we introduce the global convergence theory for  \eqref{eq: QuDRC-ALADIN} and the local convergence theory for Algorithm \ref{alg:Bilevel Consensus ALADIN} and  Algorithm \ref{alg: Decentralized Reduced BFGS Consensus ALADIN}.

\subsubsection{Global Convergence Analysis}\label{sec: Global Convergence Analysis2}

In this section, we provide the convergence analysis of \eqref{alg:ALADIN}. 
First, we introduce the two lemmas that are important for our analysis. 
Then, we prove our main result via a theorem.

\begin{lemma}[Key Properties of the Estimation Errors]
The update of the global variable estimation $\hat z_i$ of each agent $i \in \mathcal{V}$ is given by \eqref{alg:ALADIN} in \eqref{eq: local copy update}. 
According to the constraints of problem~\eqref{eq: reformulate3}, the following equation is satisfied 
    \begin{equation}\small\label{eq: error}\left\{
    \begin{aligned} 
        &\hat z_i^{[k+1]} = \frac{1}{N} \sum_{j=1}^N \Delta \left \lfloor \frac{y_j}{\Delta} \right \rfloor + \kappa_i, \;\left\| \kappa_i \right\|\leq \sqrt{n}\Delta, \\
        &\left\|z^{[k+1]}-\hat z_i^{[k+1]}\right\|\leq 2\sqrt{n}\Delta,\\[2mm]
                &\left\| \hat \lambda_i^{[k+1]} - \lambda_i^{[k+1]}
 \right\|= \left\|\rho(z^{[k+1]}-\hat z_i^{[k+1]})\right\|\leq 2\rho \sqrt{n}\Delta,
    \end{aligned}\right. 
\end{equation}
where $y_i = x_i^{[k+1]}-\frac{g_i}{\rho}$.
 
\end{lemma}
\textit{Proof.} See  \cite[Appendix]{jiang2021asynchronous}, \cite[Lemma~$1$]{rikos2023distributed2}.
 \hfill$\blacksquare$

\begin{lemma}
For the distributed consensus optimization problem presented in \eqref{eq: reformulate3}, \eqref{alg:ALADIN} establishes a relationship between the local primal update $x_i^{[k+1]}$, the local dual variables $\hat \lambda_i$ and $\hat \lambda_i^{[k+1]}$, and the global primal variable approximations $\hat z_i$ and $\hat z_i^{[k+1]}$. 
This relationship is 
    \begin{equation}\label{eq: lemma error}\small
    \begin{aligned}
        x_i^{[k+1]} = \frac{\hat \lambda_i^{[k+1]}-\hat\lambda_i^{[k]}}{2\rho}+\frac{\hat z^{[k+1]}_i+\hat z_i^{[k]}}{2}.
    \end{aligned}
    \end{equation}
\end{lemma}
\textit{Proof.} The result is similar to equation \eqref{eq: lemma3} and the proof is similar to equation \eqref{eq: lemma2}.
\hfill$\blacksquare$

Moreover, from the KKT (Karush-Kuhn-Tucker) system of problem \eqref{eq: closed form} and \eqref{eq: error}, the following formulas can be obtained,
\begin{equation}\label{eq: lam}\small
\begin{aligned}
    \sum_{i=1}^N \lambda_i^{[k]} = 0, \ \text{and} \ \left\|\sum_{i=1}^N \hat\lambda_i^{[k]}\right\| \leq 2\rho N \sqrt{n} \Delta.
\end{aligned}
\end{equation}
In \eqref{eq: lam} the first equality arises from the KKT stationarity condition for the reduced QP problem in \eqref{eq: closed form}. 
The second inequality arises from substituting the third inequality of \eqref{eq: error} into \eqref{eq: new dual} and then summing over all agents. 
For establishing the global convergence of \eqref{eq: QuDRC-ALADIN}, we introduce the following Lyapunov function, 
\begin{equation}\label{eq: lya}\small
\begin{aligned}
    \hat{\mathcal L}(\hat z,\hat \lambda) =  \rho N\left\|\hat z - z^* \right\|^2+\frac{1}{\rho} \sum_{i=1}^N \left\|\hat\lambda_i - \lambda_i^* \right\|^2 .
\end{aligned}
\end{equation}


\begin{theorem}[Global Linear Convergence of \eqref{eq: QuDRC-ALADIN}]\label{the: convergence}
Let us consider a digraph $\mathcal G= \left( \mathcal V, \mathcal E\right)$. 
Each agent $i \in \mathcal V$ has a local cost function $f_i$, and Assumptions~\ref{ass:1} and \ref{ass:2} hold.  Assume that Problem~\eqref{eq: reformulate1} is feasible and satisfies strong duality.
Each agent $i \in \mathcal V$ in the network executes \eqref{alg:ALADIN} for solving the consensus optimization problem in~\eqref{eq: reformulate3} in a decentralized fashion. 
Given parameter $\rho>0$, 
during the operation of \eqref{alg:ALADIN} there always exists a $\delta>0$ such that 
\begin{equation}\label{eq: condition}\small
\begin{aligned}
     \delta \hat{\mathcal L}\left(\hat z^{[k+1]}, \hat\lambda^{[k+1]}\right)\leq 4 \sum_{i=1}^N \mu_i\left\|x_i^{[k+1]} -z^* \right\|^2 , 
\end{aligned}
\end{equation} 
where $\hat z = \hat z_i, \forall i \in \mathcal{V}$. 
From \eqref{eq: condition}, we have that during the operation of \eqref{alg:ALADIN} the following inequality is satisfied 
\begin{equation}\label{eq: reault1}\small
\begin{aligned}
    \hat{\mathcal L} \left(\hat z^{[k+1]}, \hat\lambda^{[k+1]}\right) \leq \frac{1}{1+ \delta}\hat{\mathcal L} \left(\hat z^{[k]}, \hat\lambda^{[k]}\right) + \frac{4}{1+ \delta} \mathcal O(N\sqrt{n}\Delta),
\end{aligned}
\end{equation} 
where 
 $\Delta$ is the utilized quantization level.
\end{theorem}
\textit{Proof.} See Appendix \ref{APP: global convergence error1}. \hfill$\blacksquare$

\begin{remark}[Quantization Error and Resulting Trade-offs]  \label{remark: tradeoff} 
    Focusing on \eqref{eq: reault1} of Theorem~\ref{the: convergence}, the term $\frac{4}{1+ \delta} \mathcal O(N\sqrt{n}\Delta)$ represents the quantization error introduced from Algorithm~\ref{alg:QuAS}. 
    As we will see later in Section~\ref{sec: numerical}, this error causes agents to converge to a $\Delta$-dependent neighborhood of the optimal solution. 
    While decentralized approaches to progressively refine $\Delta$ can enhance solution precision \cite{rikos2023distributed}, this typically incurs higher communication overhead in terms of bits per message compromising our algorithm's communication efficiency.
    In contrast, employing quantizers with base-shifting capabilities \cite{rikos2024finite_bit_zooming} allows for maintaining high communication efficiency while still enabling agents to approximate the optimal solution with greater precision.
    This latter strategy however, results in a trade-off as it may reduce the convergence speed of Algorithm~\ref{alg:QuAS}.
\end{remark} 

\begin{remark}[Upper Bound of Error Accumulation]
\label{remark error accumulation}
In Theorem~\ref{the: convergence} we established that \eqref{alg:ALADIN} enables agents to compute the optimal solution for the consensus optimization problem in \eqref{eq: reformulate3} with a global linear convergence rate. 
However, the term $\mathcal O(N \Delta)$ in \eqref{eq: reault1} represents the error due to quantized communication among agents. 
To derive an upper bound on the error accumulation during \eqref{alg:ALADIN}, during each time step $k$ we have 
\begin{equation}\label{eq: iteration}\small
    \begin{split}
        &\rho N \left\|\hat z^{[k]} -z^*\right\|^2\leq \hat{\mathcal L} \left(\hat z^{[k]}, \hat\lambda^{[k]} \right) \\
       \overset{\eqref{eq: reault1}}{\leq}  &\frac{1}{1+\delta} \hat{\mathcal L} \left(\hat z^{[k-1]}, \hat\lambda^{[k-1]} \right)+\frac{1}{1+\delta} \mathcal  O(4N\sqrt{n}\Delta)\\
       \vdots\\
        \overset{\eqref{eq: reault1}}{\leq} &\left(\frac{1}{1+\delta} \right)^{k-1} \hat{\mathcal L}\left(\hat z^{[1]}, \hat\lambda^{[1]} \right)+ V(\delta,k)\; \mathcal O\left(\frac{4}{\delta}N\sqrt{n}\Delta\right),
    \end{split}
\end{equation}
where $\left(\frac{1}{1+\delta} \right)^{k-1}<1$ and $V(\delta,k) =  \left(1- \left( \frac{1}{1+\delta}\right)^{k-1} \right)$. 
This analysis demonstrates that after $k$ iterations of \eqref{alg:ALADIN}, the error accumulation from \eqref{eq: reault1} is upper-bounded by $V(\delta,k)\; \mathcal O\left(\frac{4}{\delta}N\sqrt{n}\Delta\right)$. 
This upper bound is directly influenced by the quantization level $\Delta$ employed in our proposed algorithm.
\end{remark}

\subsubsection{Local Convergence Analysis}\label{sec: Local Convergence Analysis2}

In this section, we establish the local convergence properties of Algorithm \ref{alg:Bilevel Consensus ALADIN} and Algorithm \ref{alg: Decentralized Reduced BFGS Consensus ALADIN}. Prior to presenting the convergence theory, we introduce the following assumption, which ensures the convergence of both algorithms.

\begin{assumption}[Lipschitz Continuous]\label{ass:3} 
The local cost function $f_i$ of each agent $i \in \mathcal{V}$ is closed, proper, $L$-smooth and  convex. 
Specifically, for each local cost function $f_i$, for every $x_\alpha,x_\beta \in \mathbb R^n$, there exist a \emph{Lipschitz-continuity} constant $L_i>0$ (see \cite[equation (51)]{ling2015dlm}) such that \eqref{eq: lip} holds.
\end{assumption}
Assumption~\ref{ass:3} ensures the Lipschitz continuity of the gradients in \eqref{eq: lip}, which is employed in the following theorem to derive an upper bound on the estimation error. This is a standard and widely adopted assumption in optimization theory.

Now, we provide the local convergence theory of Algorithm \ref{alg:Bilevel Consensus ALADIN} by showing the follow theorem.
\begin{theorem}[Local Convergence  of Algorithm \ref{alg:Bilevel Consensus ALADIN}]\label{them: local convergence error1}
Consider a directed graph $\mathcal{G} = (\mathcal{V}, \mathcal{E})$ where each agent $i \in \mathcal{V}$ possesses a local cost function $f_i$. Under Assumptions~\ref{ass:1},~\ref{ass:02}, and~\ref{ass:3}, each agent executes Algorithm~\ref{alg:Bilevel Consensus ALADIN} in a decentralized manner to solve the consensus optimization problem~\eqref{eq: reformulate3}. Provided that the variables $x_i$ and $z$ are initialized within a neighborhood of a local optimum $z^*$, the iterations of Algorithm~\ref{alg:Bilevel Consensus ALADIN} achieve local convergence at a linear rate.
	\end{theorem}
	\textbf{Proof:} See Appendix \ref{APP: local convergence error 1}. \hfill$\blacksquare$

A similar local convergence result holds for Algorithm~\ref{alg: Decentralized Reduced BFGS Consensus ALADIN}; the proof is provided in Appendix~\ref{APP: local convergence error 2}.



\section{Numerical Simulation}\label{sec: numerical}


Numerical experiments are conducted on both convex and non-convex instances of the distributed consensus optimization In this section, we use problem \eqref{eq: reformulate3} to evaluate the performance of the proposed algorithms and to highlight their improvements over existing distributed optimization methods. All simulations are implemented using \texttt{CasADi-v3.5.5} in conjunction with the \texttt{IPOPT} solver~\cite{Andersson2019}.

\subsection{Convex Case}

For the convex setting, we consider a distributed consensus least-squares problem following the formulation in~\cite{Du2023_Arxiv}:
\begin{equation}\label{eq: convex example}\small
\begin{aligned}
\min_{x_i, \forall i \in \mathcal{V}} \quad  \sum_{i=1}^{N} \frac{1}{2}\|x_i - \zeta_i\|^2 \qquad
\text{s.t.} \quad & x_i = z, \quad \forall i \in \mathcal{V}.
\end{aligned}
\end{equation}
The problem is defined over a randomly generated directed graph with $N = 20$ agents, where $x_i, z \in \mathbb{R}^{10}$. The local data $\zeta_i$ are sampled from a Gaussian distribution $\mathcal{N}(0,1)$. In this setup, the problem involves $210$ primal variables and $200$ dual variables. The convergence performance is evaluated by plotting the consensus error $\sum_{i=1}^N \|x_i^{[k]} - x_i^*\|_1$, where $x^* = [(z^*)^\top, (z^*)^\top, \dots, (z^*)^\top]^\top\in \mathbb R^{Nn}$ denotes the optimal solution of \eqref{eq: convex example}. According to \eqref{eq: lemma} and \eqref{eq: lemma error}, the convergence behavior of the primal variables $x_i$ coincides with that of the global variable $z$ and the dual variables $\lambda_i$; therefore, the same convergence properties apply to all variables considered throughout the algorithms proposed in this paper.

\begin{figure}[ht]
	\centering
\includegraphics[width=0.5\textwidth,height=0.3\textheight]{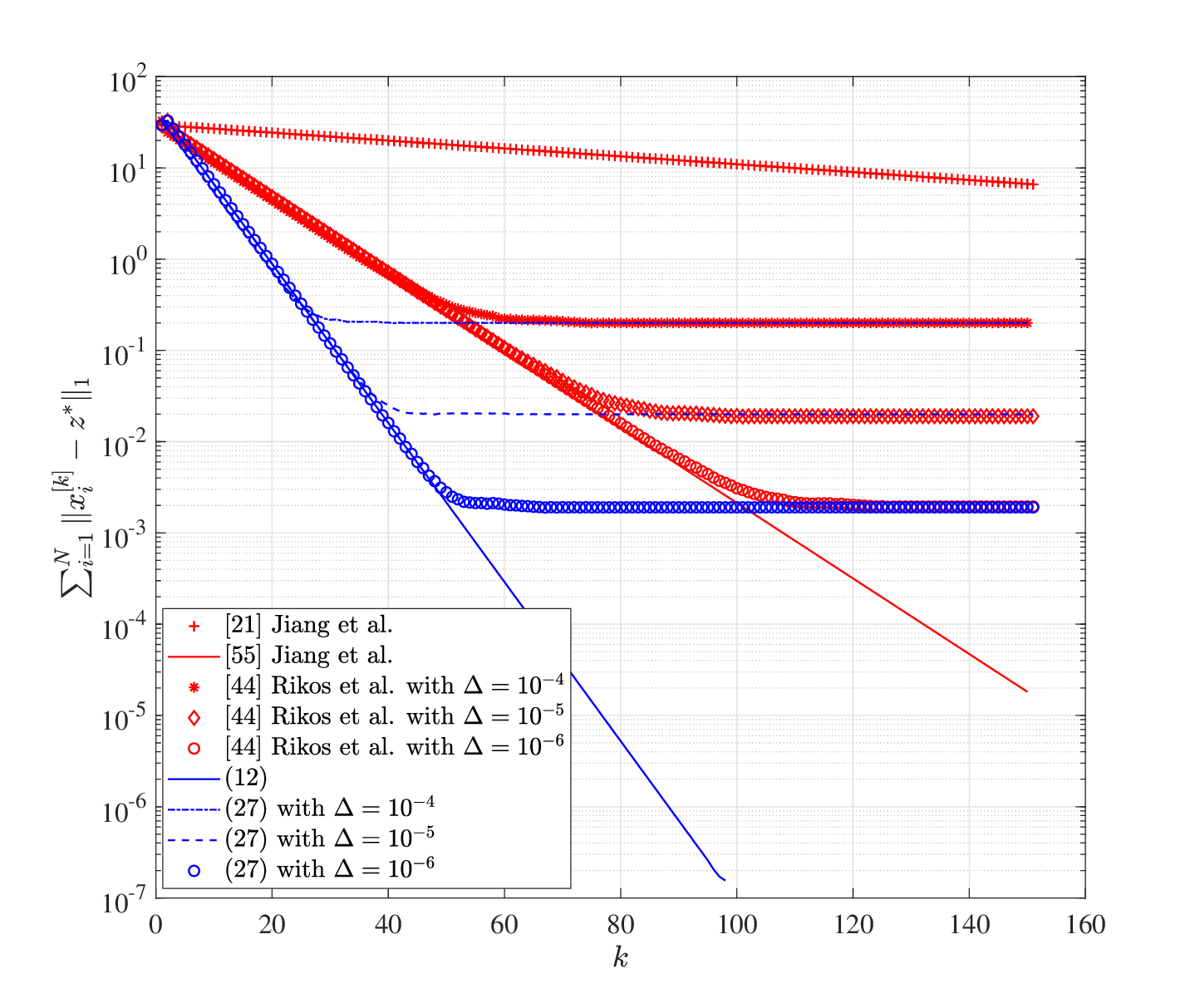}
	\caption{Comparison of convergence performance among Pull-FTERC~\cite{jiang2022fast}, AsyAD-ADMM~\cite{jiang2021asynchronous}, QuAsyADMM~\cite{rikos2023asynchronous}, the proposed first-order C-ALADIN  \eqref{eq: RCA}, and its quantized decentralized version \eqref{eq: QuDRC-ALADIN} under different quantization levels $\Delta = 10^{-4}, 10^{-5}, 10^{-6}$.}
	\label{fig: convex first comparison1}
\end{figure}

Fig.~\ref{fig: convex first comparison1} compares the convergence behavior of Pull-FTERC~\cite{jiang2022fast}, AsyAD-ADMM~\cite{jiang2021asynchronous}, QuAsyADMM~\cite{rikos2023asynchronous}, the proposed first-order C-ALADIN \eqref{eq: RCA}, and its quantized decentralized version \eqref{eq: QuDRC-ALADIN} under different quantization levels $\Delta = 10^{-4}, 10^{-5}, 10^{-6}$. The results demonstrate that ADMM-type algorithms converge faster than Pull-FTERC~\cite{jiang2022fast} due to their utilization of dual information. QuAsyADMM converges to a neighborhood of the optimal solution of problem \eqref{eq: convex example}, with smaller $\Delta$ values yielding higher accuracy, as theoretically characterized in~\cite{rikos2023asynchronous}. The first-order C-ALADIN achieves accelerated convergence compared to Consensus ADMM by utilizing more gradient information per iteration, as discussed in \cite[Appendix]{Du2025ACC}. The fully decentralized quantized version of first-order C-ALADIN offers a favorable trade-off between communication efficiency and solution accuracy: increasing $\Delta$ reduces bandwidth consumption at the expense of precision, while the coordinator‑free design maintains scalability and robustness in large‑scale networks.


\begin{figure}[ht]
	\centering
\includegraphics[width=0.45\textwidth,height=0.24\textheight]{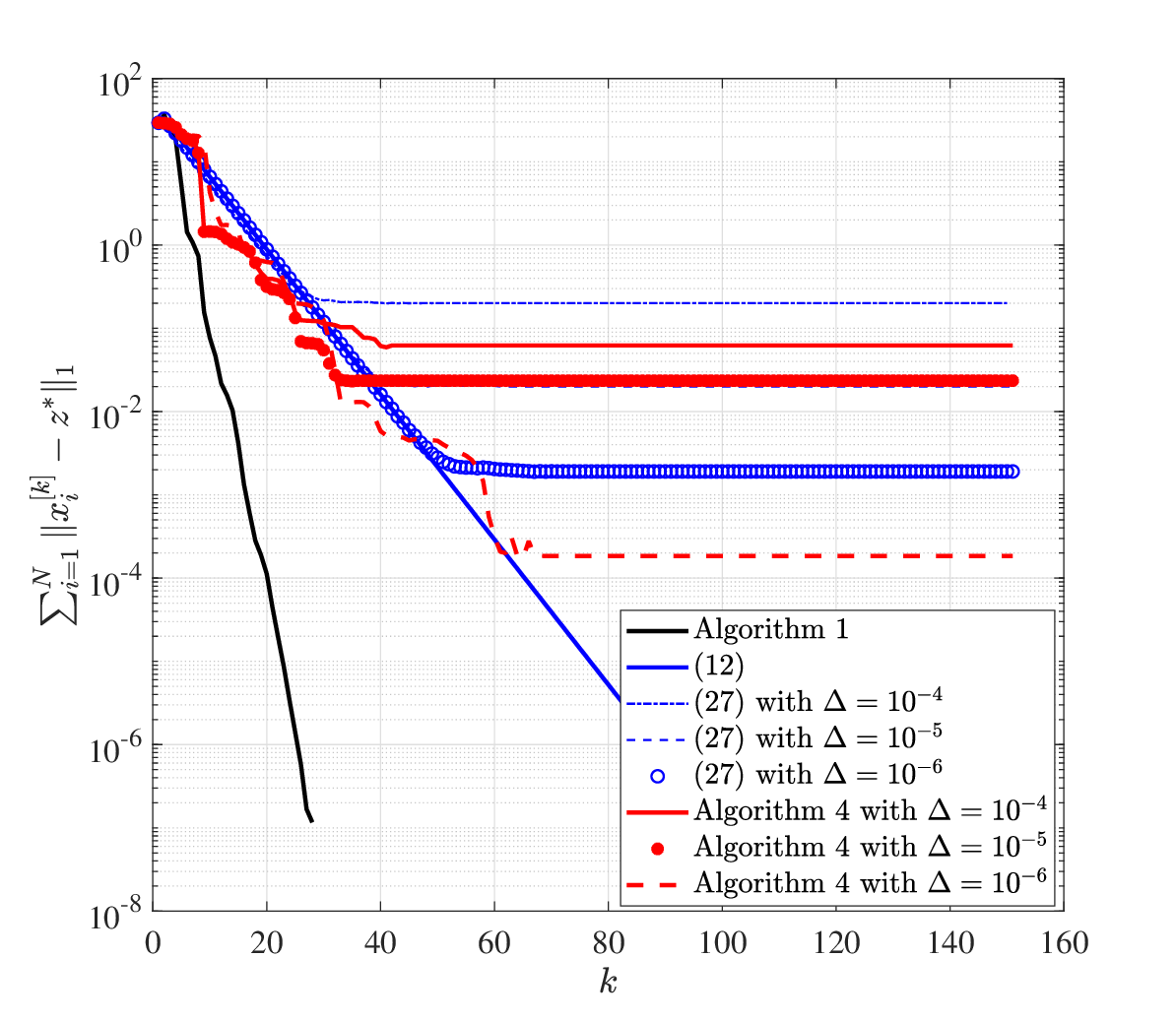}
	\caption{Convergence Algorithm~\ref{alg: Decentralized Reduced BFGS Consensus ALADIN}among Algorithm~\ref{alg:BFGS ALADIN2}, the first-order C-ALADIN in \eqref{eq: RCA}, its quantized decentralized version \eqref{eq: QuDRC-ALADIN}, and the decentralized approximate second-order C-ALADIN (Algorithm~\ref{alg: Decentralized Reduced BFGS Consensus ALADIN}) under different quantization levels $\Delta = 10^{-4}, 10^{-5}, 10^{-6}$.}
	\label{fig: convex second comparison3}
\end{figure}

Fig.~\ref{fig: convex second comparison3} illustrates the convergence behavior of the second-order C-ALADIN (Algorithm~\ref{alg:BFGS ALADIN2}), the first-order C-ALADIN \eqref{eq: RCA}, its quantized decentralized version \eqref{eq: QuDRC-ALADIN}, and the decentralized approximate second-order C-ALADIN (Algorithm~\ref{alg: Decentralized Reduced BFGS Consensus ALADIN}) under quantization levels $\Delta = 10^{-4}$, $10^{-5}$, and $10^{-6}$. The results show that Algorithm~\ref{alg:BFGS ALADIN2}, which incorporates second-order information, achieves significantly faster convergence. In this convex setting, the decentralized approximate second-order variant does not exhibit a notable advantage over the quantized decentralized first-order variant in convergence speed; its superiority becomes even more pronounced in the subsequent non-convex case study.

\begin{figure}[ht]
	\centering
\includegraphics[width=0.37\textwidth,height=0.2\textheight]{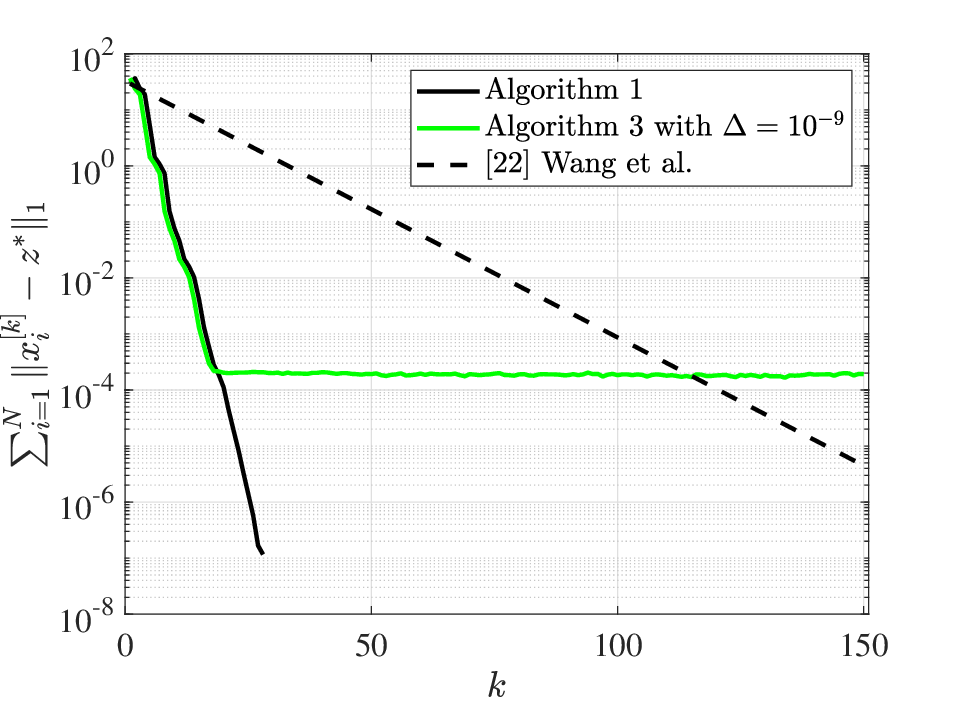}
	\caption{Convergence comparison among Algorithm~\ref{alg:BFGS ALADIN2}, Algorithm~\ref{alg:Bilevel Consensus ALADIN} with $\Delta = 10^{-9}$, and GIANT~\cite{wang2018giant}.}
	\label{fig: convex second comparison4}
\end{figure}

Fig.~\ref{fig: convex second comparison4} illustrates the convergence behavior of GIANT~\cite{wang2018giant}, the second-order C-ALADIN (Algorithm~\ref{alg:BFGS ALADIN2}), and its bi-level quantized decentralized version (Algorithm~\ref{alg:Bilevel Consensus ALADIN}) with quantization level $\Delta = 10^{-9}$. Due to the bi-level structure of Algorithm~\ref{alg:Bilevel Consensus ALADIN}, its convergence is highly sensitive to the accuracy of the inner consensus QP solver; therefore, only the case with $\Delta = 10^{-9}$ is shown. In this setting, Algorithm~\ref{alg:Bilevel Consensus ALADIN} converges nearly as fast as Algorithm~\ref{alg:BFGS ALADIN2}, reaching a neighborhood of the optimal solution. For comparison, the centralized GIANT method, which is based on primal decomposition, is included. 

\subsection{Non-convex Case}

For the same numerical example as in our conference version~\cite{Du2025ACC}, we consider a non-convex sensor allocation application adapted from~\cite[Section 8]{Houska2016}. The resulting non-convex distributed consensus optimization problem is formulated as
\begin{equation}\label{eq: non-convex problem}\small
		\begin{aligned}
			\min_{x_i,z,\forall i \in \mathcal I}&\quad \sum_{i=1}^{N}\left(\frac{1}{2}\left(\|x_i^\alpha- \zeta_i^\alpha\|^2+\|x_i^\beta- \zeta_i^\beta\|^2\right)\right.\\
			&\quad\quad\;\;\left.+\sum_{j=1}^{10}\left(\left(x_{i}^\alpha[j]-x_{i}^\beta[j]\right)^2-\zeta_{i}^\sigma[j]\right)^2\right) \\ \quad\mathrm{s.t.}\quad&\quad x_i = z,
		\end{aligned}
\end{equation}
where $x_i = [(x_i^\alpha)^\top, (x_i^\beta)^\top]^\top$, and $(\cdot)[j]$ denotes the $j$-th component of the corresponding vector. All measured data $\zeta_i^\alpha,\zeta_i^\beta, \zeta_i^\sigma$ and parameter settings are consistent with those used in the convex example~\eqref{eq: convex example}. The first-order variants of C-ALADIN diverge in this non-convex setting and are therefore omitted from the comparison. 

\begin{figure}[ht]
	\centering
\includegraphics[width=0.46\textwidth,height=0.26\textheight]{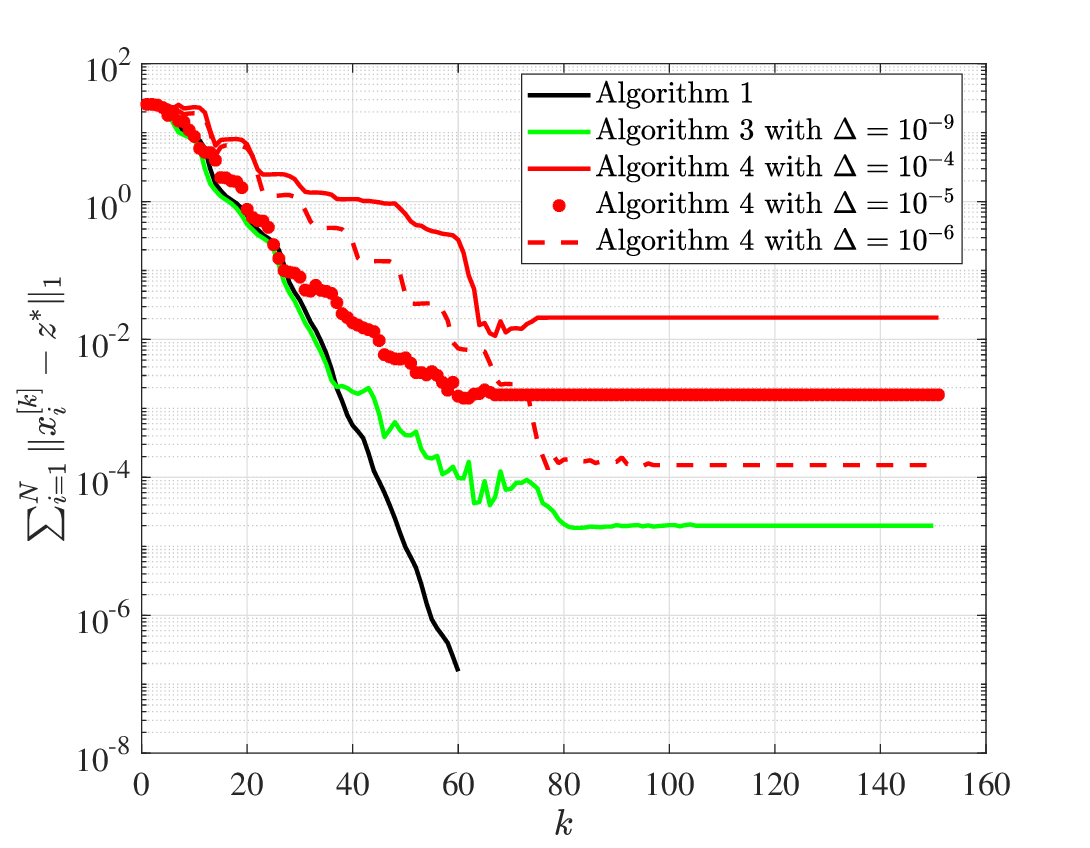}
	\caption{Convergence comparison among Algorithm~\ref{alg:BFGS ALADIN2}, Algorithm~\ref{alg:Bilevel Consensus ALADIN} with $\Delta = 10^{-9}$, and Algorithm~\ref{alg: Decentralized Reduced BFGS Consensus ALADIN} under different quantization levels $\Delta = 10^{-4}, 10^{-5}, 10^{-6}$.}
	\label{fig: comparison}
\end{figure}

Fig.~\ref{fig: comparison} illustrates the convergence behavior of the second-order C-ALADIN (Algorithm~\ref{alg:BFGS ALADIN2}), its bi-level quantized decentralized version (Algorithm~\ref{alg:Bilevel Consensus ALADIN}) with $\Delta = 10^{-9}$, and the decentralized approximate second-order C-ALADIN (Algorithm~\ref{alg: Decentralized Reduced BFGS Consensus ALADIN}) under quantization levels $\Delta = 10^{-4}, 10^{-5}, 10^{-6}$ for the non-convex application \eqref{eq: non-convex problem}. All second-order C-ALADIN variants demonstrate local convergence to a neighborhood of a stationary point in this non-convex setting. Notably, Algorithm~\ref{alg:Bilevel Consensus ALADIN} and Algorithm~\ref{alg: Decentralized Reduced BFGS Consensus ALADIN} rely exclusively on decentralized quantized averaging for network communication, achieving a balance between communication efficiency and solution accuracy. For the latter algorithm, a sensitivity to quantization noise is observed; it is therefore omitted for the case $\Delta = 10^{-4}$ due to numerical instability. Although GIANT~\cite{wang2018giant} is also a second-order method, it fails to converge on this non-convex problem; it is not included in the comparison.

\begin{remark}[On Solving \eqref{eq: reformulate1} with typical ALADIN]
This work addresses distributed consensus optimization for both convex and non-convex problems. While such problems can be reformulated as resource allocation and solved using typical ALADIN~\eqref{eq: T-ALADIN} (as shown in \cite{Houska2016}), we note that typical ALADIN involves substantial communication and computational overhead. More importantly, typical ALADIN exhibits essentially equivalent convergence behavior to C-ALADIN, differing only in the linear constraint formulation in \eqref{eq: DOPT_G2}. Given this equivalence and the significant implementation complexity of typical ALADIN, we omit numerical comparisons in this study. For implementation details of typical ALADIN on Problem~\eqref{eq: reformulate1}, we refer readers to \cite[Section 12]{aadhithya2023learning}.
\end{remark}

\begin{remark}[On the Effect of Penalty Parameter $\rho$]
The numerical study in \cite[Section 5.1]{Houska2021} demonstrates the empirical effect of the penalty parameter $\rho$ in C-ALADIN \eqref{alg:ALADIN} on convergence rates for convex problems under Assumption~\ref{ass:2}. While convergence speed is observed to depend on $\rho$, a theoretical characterization of this relationship remains an open problem. A thorough investigation of this dependency, including both theoretical analysis and comprehensive numerical comparison, is beyond the scope of this paper and represents an important direction for future research.
\end{remark}

\section{Conclusion}\label{sec: conclusion}
This paper proposes Consensus ALADIN (C-ALADIN) and its variants for efficiently solving distributed consensus optimization problems. Within this framework, we design an efficient structure that enhances both communication and computational efficiency. Building on this framework, we develop centralized and quantized decentralized variants of C-ALADIN, each of which can be further classified into first-order and second-order implementations. We establish the convergence properties of C-ALADIN for both convex and non-convex settings and observe that, under decentralized operation with quantized communication, the proposed algorithms converge to a neighborhood of the optimal solution. Finally, numerical simulations are presented to validate the performance of our algorithms and to highlight their advantages over existing methods in the literature.


\appendices

\section{PROOF OF THEOREM \ref{theorem 2}}\label{APP: theorem 2}
					
Focusing on equation \eqref{eq: mu}, the following symmetric inequalities can be obtained 
\begin{equation}\small
\label{eq: strongly convex2} \left\{
\begin{aligned}
     & f_i(x_\alpha)+\nabla f_i(x_\alpha)^\top (x_\beta-x_\alpha) +\frac{\mu_i}{2}\|x_\beta-x_\alpha\|^2\leq f_i(x_\beta)\\
    & f_i(x_\beta)+\nabla f_i(x_\beta)^\top (x_\alpha-x_\beta) +\frac{\mu_i}{2}\|x_\alpha-x_\beta\|^2\leq f_i(x_\alpha).
\end{aligned}\right.
\end{equation}
By adding the two inequalities in \eqref{eq: strongly convex2} and setting $x_\alpha = x_i^{[k+1]}$, $x_\beta = z^*$, we obtain
\begin{equation}\label{eq: strongly convex0}\small
    \mu_i \left\| x_i^{[k+1]} - z^* \right\|^2 \leq \left( \nabla f_i(x_i^{[k+1]}) - \nabla f_i(z^*) \right)^\top \left( x_i^{[k+1]} - z^* \right).
\end{equation}
Under Assumption~\ref{ass:2}, each $f_i$ is smooth; hence $g_i = \nabla f_i(x_i^{[k+1]})$ and $g_i^* = \nabla f_i(z^*)$. Summing the $N$ inequalities in \eqref{eq: strongly convex0} yields
\begin{equation}\label{eq: strongly convex3}\small
    \sum_{i=1}^N \mu_i \left\| x_i^{[k+1]} - z^* \right\|^2 \leq \sum_{i=1}^N \left( g_i - g_i^* \right)^\top \left( x_i^{[k+1]} - z^* \right).
\end{equation}
Therefore, from \cite[equation (23) - (26)]{Du2025ACC}, equation \eqref{eq: strongly convex3} can be represented as,
	\begin{equation}\label{eq: linear rate}\small
		\begin{aligned}
			 \sum_{i=1}^{N}\mu_i\left\| x_i^{[k+1]} -z^*\right\|^2 
			\leq &\sum_{i=1}^{N} \left(x_i^{[k+1]}-z^* \right)^\top \left( g_i -g_i^* \right)\\
             = &	\frac{1}{4}\mathscr{L}\hspace{-0.5mm}\left(\hspace{-0.5mm}z^{[k]},\lambda^{[k]}\hspace{-1mm}\right)\hspace{-1mm}-\hspace{-1mm}\frac{1}{4}\mathscr{L}\left(z^{[k+1]},\lambda^{[k+1]}\right).
		\end{aligned}
	\end{equation}

We now prove the existence of $\delta$ in \eqref{eq: dela}. The updates of $z^{[k+1]}$ and $\lambda_i^{[k+1]}$ from \eqref{eq: closed form} satisfy the following relations:
\begin{equation*}\label{eq: smooth1}\small
    \begin{split}
        z^{[k+1]} - z^* 
       \overset{\eqref{eq: global z}}{=} & \left( \sum_{i=1}^N B_i  \right)^{-1} \hspace{-2mm}\left( \sum_{i=1}^N B_i x_i^{[k+1]}  - B_i z^* - g_i + g_i^*  \right),\\
       \lambda_i^{[k+1]} - \lambda_i^*
         \overset{\eqref{eq: dual}}{=}& B_i(x_i^{[k+1]} - z^{[k+1]}) - g_i + g_i^*\\
         =& B_i(x_i^{[k+1]} - z^*) + B_i(z^* - z^{[k+1]}) - g_i + g_i^*.
    \end{split}
\end{equation*}
Note that $\sum_{i=1}^N g_i^* = 0$ follows from the optimality conditions of \eqref{eq: reformulate1}.
Using the Lipschitz gradient assumption $\|g_i - g_i^*\| \le L_i \|x_i^{[k+1]} - z^*\|$ (Assumption~\ref{ass:2}) and the triangle inequality, there exist constants $C_z < \infty$ and $C_\lambda < \infty$ such that
\begin{subequations}\label{eq: upper_bounds}
\begin{align}
\sum_{i=1}^N \left\|z^{[k+1]} - z^*\right\|_{B_i}^2 &\le C_z \sum_{i=1}^N \left\|x_i^{[k+1]} - z^*\right\|_{B_i}^2, \label{eq: upper_z}\\
\sum_{i=1}^N \left\|\lambda_i^{[k+1]} - \lambda_i^*\right\|_{B_i^{-1}}^2 &\le C_\lambda \sum_{i=1}^N \left\|x_i^{[k+1]} - z^*\right\|_{B_i}^2. \label{eq: upper_lambda}
\end{align}
\end{subequations}
These constants are uniform upper bounds over all iterations. Adding \eqref{eq: upper_z} and \eqref{eq: upper_lambda} yields
\begin{equation}\label{eq: Lyapunov_upper}
\mathscr{L}(z^{[k+1]},\lambda^{[k+1]}) \le (C_z + C_\lambda) \sum_{i=1}^N \left\|x_i^{[k+1]} - z^*\right\|_{B_i}^2.
\end{equation}
Combining \eqref{eq: Lyapunov_upper} with the strong convexity inequality $4\sum_{i=1}^N \mu_i \left\|x_i^{[k+1]} - z^*\right\|^2 \ge 4\mu_{\min} \sum_{i=1}^N \left\|x_i^{[k+1]} - z^*\right\|^2$, where $\mu_{\min} = \min_i \mu_i$, and the fact that $\left\|x_i^{[k+1]} - z^*\right\|_{B_i}^2 \le \|B_i\| \left\|x_i^{[k+1]} - z^*\right\|^2$, we obtain \eqref{eq: dela} with 
$
\delta = \frac{4\mu_{\min}}{(C_z + C_\lambda) \max_i \|B_i\|} > 0.
$
Hence the constant $\delta$ in \eqref{eq: dela} exists.

	Combining the result of equation \eqref{eq: linear rate} 
	and \eqref{eq: dela}, \eqref{eq: descent} is then proved.
	Based on the one-step linear convergence rate in \eqref{eq: descent}, the global linear convergence in \eqref{eq: linear convergence proof1} can be established through cumulative multiplication of the contraction factors. Theorem \ref{theorem 2} is proved.

	\section{Proof of Theorem \ref{the: convergence}}\label{APP: global convergence error1}
Before starting the proof, let us recall the following equality 
\begin{equation}\label{eq: well-known equality}\begin{aligned}
    2\left(\alpha-\gamma \right)^\top \left(\alpha-\beta \right) =
    \left\|\alpha-\gamma \right\|^2 - \left\|\beta-\gamma \right\|^2+\left\|\alpha-\beta \right\|^2,
\end{aligned}
\end{equation}
for $\alpha, \beta, \gamma \in \mathbb R^n$.  
Moreover, since the convergence of \eqref{eq: BFGS ALADIN} has been established in Appendix \ref{APP: theorem 2}, demonstrating that $\left\|z - z^* \right\| $ is bounded, we follow the analysis in \cite{jiang2021asynchronous} by defining the finite positive scalar $0 < M_z < \infty$ to simplify our subsequent analysis. This ensures that $\left\|z_i - z^* \right\| \leq M_z$ holds for every agent $i \in \mathcal{V}$.

We start from equation \eqref{eq: strongly convex3}. For \eqref{eq: QuDRC-ALADIN}, the right hand side of \eqref{eq: strongly convex3} can be represented as,
    \begin{align}\label{eq: simplify}\small
        & \sum_{i=1}^N\left (g_i -g_i^*\right)^\top \left(x_i^{[k+1]} -z^* \right)\\
        \overset{\eqref{eq: new gradient}}{=}& \sum_{i=1}^N\left (\rho\left(\hat z_i^{[k]}-x_i^{[k+1]}\right)-\hat\lambda_i^{[k]} +\lambda_i^*\right)^\top \left(x_i^{[k+1]} -z^* \right)\notag\\
        \overset{\eqref{eq: lemma error}}{=}& \sum_{i=1}^N\left (\rho\left(
        \frac{\hat z_i^{[k]}-\hat z^{[k+1]}_i}{2}
        -\frac{\hat \lambda_i^{[k+1]}-\hat\lambda_i^{[k]}}{2\rho}\right)-\hat \lambda_i^{[k]} +\lambda_i^*\right)^\top\notag \\
        &\hspace{2.5mm}\quad\left(\frac{\hat \lambda_i^{[k+1]}-\hat\lambda_i^{[k]}}{2\rho}+\frac{\hat z^{[k+1]}_i+\hat z_i^{[k]}}{2} -z^* \right).\notag
    \end{align}
We now split the right-hand side of \eqref{eq: simplify} into the following components (a), (b) (c), (d), (e). 
Then, we analyze each component separately. 
{\small \begin{align}\small
        (a) &\quad         \frac{1}{4}\sum_{i=1}^N \left(\hat z_i^{[k]} - \hat z_i^{[k+1]} \right)^\top \left( \hat\lambda_i^{[k+1]} - \hat\lambda_i^{[k]}\right)\notag\label{eq: 1}\\
          & = \frac{1}{4}\sum_{i=1}^N \left(\left( z_i^{[k]}-z^*\right) - \left( z_i^{[k+1]}-z^* \right)\right.\\
          &\hspace{1.3cm}+\left.\left( \hat z_i^{[k]}-z_i^{[k]} \right) -\left( \hat z_i^{[k+1]}-z_i^{[k+1]} \right) \right)^\top \notag\\
           &\;\qquad\qquad\left(\left( \lambda_i^{[k]}-\lambda_i^*\right) - \left( \lambda_i^{[k+1]}-\lambda_i^* \right)\right.\notag \\
           &\hspace{1.3cm}\left.+\left( \hat \lambda_i^{[k]}-\lambda_i^{[k]} \right) -\left( \hat \lambda_i^{[k+1]}-\lambda_i^{[k+1]} \right) \right)\notag\\
          &  \overset{\eqref{eq: error},\eqref{eq: lam}}{\leq} 2N\rho \left(\sqrt{n}M_z\Delta+2n\Delta^2 \right). \notag\\
           (b) &\quad  \sum_{i=1}^N \left( \frac{\hat \lambda_i^{[k+1]} - \hat \lambda_i^{[k]}}{2\rho}\right)^\top\left( \frac{\hat \lambda_i^{[k]} -\hat \lambda_i^{[k+1]}}{2} +\lambda_i^* -\hat\lambda_i^{[k]}\right)\notag\\
       &     \overset{\eqref{eq: well-known equality}}{=} \frac{1}{4\rho}\sum_{i=1}^N \left(\left\|\hat \lambda_i^{[k]}-\lambda_i^*\right\|^2 -\left\|\hat \lambda_i^{[k+1]}-\lambda_i^*\right\|^2\right).\label{eq: 23}\\
       (c)& \quad \frac{\rho}{2}\sum_{i=1}^N \left( \frac{\hat z_i^{[k+1]}+\hat z_i^{[k]}}{2}-z^*\right)^\top\left( \hat z_i^{[k]} - \hat z_i^{[k+1]} \right)\label{eq: 4+7}\\
         &   \overset{\eqref{eq: well-known equality}}{=} \frac{\rho}{4} \sum_{i=1}^N \left( \left\|\hat z_i^{[k]} - z^* \right\|^2 -\left\|\hat z_i^{[k+1]} - z^* \right\|^2 \right).\notag\\
         (d)&   \quad      \sum_{i=1}^N \left(\frac{\hat z_i^{[k+1]}+\hat z_i^{[k]} -2z^*}{2} \right)^\top \left(\frac{\hat \lambda_i^{[k]} - \hat\lambda_i^{[k+1]}}{2} \right)\label{eq: 5+8}\\
         &  \overset{\eqref{eq: error},\eqref{eq: lam}}{\leq} 2N\rho \left(\sqrt{n}M_z\Delta+2n\Delta^2 \right).\notag\\
         (e)&\quad     \sum_{i=1}^N \left(\frac{\hat z_i^{[k+1]}+\hat z_i^{[k]} -2z^*}{2} \right)^\top \left(\lambda_i^* -\hat \lambda_i^{[k]} \right)\label{eq: 6+9}\\
           &\overset{\eqref{eq: error},\eqref{eq: lam}}{\leq} 2N\rho \left(\sqrt{n}M_z\Delta+2n\Delta^2 \right).\notag
    \end{align}}

Note here that the analysis in \eqref{eq: 5+8} and \eqref{eq: 6+9} followed the same methodology as the one in \eqref{eq: 1}, but we omitted the details for space considerations. 
Considering now that $\hat z= \hat z_i$ for every agent $i \in \mathcal{V}$, and summing \eqref{eq: 1} -- \eqref{eq: 6+9}, we have that \eqref{eq: strongly convex3} becomes 
\vspace{-1mm}
\begin{equation}\label{eq: end}\small
\begin{split}
    &\sum_{i=1}^N \mu_i \left\| x_i^{[k+1]} -z^*\right\|^2
    \\
    \leq& \frac{1}{4} \hat{\mathcal L} \left(\hat z^{[k]}, \hat\lambda^{[k]}\right)-\frac{1}{4}\hat{\mathcal L} \left(\hat z^{[k+1]}, \hat\lambda^{[k+1]}\right)+ \mathcal O(N\sqrt{n} \Delta),
\end{split}
\end{equation}
where $\mathcal O(N \sqrt{n} \Delta) = 6\rho  M_z  N\sqrt{n}\Delta$. 
Note that in \eqref{eq: end} we omitted the higher-order terms of $\Delta$, as they are significantly smaller than the first-order terms and can be considered negligible. \ADD{Note that the existence of $\delta$ in \eqref{eq: condition} is the same as that in \eqref{eq: dela}, and therefore it is not shown here.}
 Finally, combining \eqref{eq: end} with \eqref{eq: condition}, we can obtain the following equation,
\begin{equation}\label{eq: range}\small
\begin{aligned}
   & \frac{\delta}{4}\hat{\mathcal L} \left(\hat z^{[k+1]}, \hat\lambda^{[k+1]} \right)\\
    \leq& \frac{1}{4}\hat{\mathcal L}\left(\hat z^{[k]}, \hat \lambda^{[k]} \right) - \frac{1}{4}\hat{\mathcal L}\left(\hat z^{[k+1]}, \hat \lambda^{[k+1]} \right)+ \mathcal O(N\sqrt{n}  \Delta),
\end{aligned}
\end{equation}
which is equivalent to \eqref{eq: reault1}.
This concludes the proof of our theorem.

            


\section{Proof of Theorem \ref{them: local convergence error1}}\label{APP: local convergence error 1}
The inner level of Algorithm~\ref{alg:Bilevel Consensus ALADIN} corresponds to \eqref{eq: QuDRC-ALADIN}, and thus its convergence properties follow directly from Theorem~\ref{the: convergence}.  
		We now establish the local convergence of the outer-level iterations in Algorithm~\ref{alg:Bilevel Consensus ALADIN}. Our analysis follows a similar approach to \cite[Theorem~3]{Du2025ACC}, \cite[Section 4.1]{Houska2021} and \cite{Du2019}, and for completeness, we provide the detailed derivation below.
        
Assume that the initial state $x_i$ of each agent lies in a small neighborhood of $z^*$. From the first-order optimality conditions of Problem~\eqref{eq: NLP}, we obtain:
\begin{equation}\label{eq: first order}
    \begin{cases}
        \nabla f_i(x_i^{[k+1]}) + \hat{\lambda}_i^{[k]} + \rho(x_i^{[k+1]} - \hat{z}_i^{[k]}) = 0, \\
        \nabla f_i(z^{*}) + \lambda_i^* = 0.
    \end{cases}
\end{equation}
Subtracting the two equations yields:
\begin{equation}\label{eq: important equality}\small
    \nabla f_i(z^{*}) - \nabla f_i(x_i^{[k+1]}) + \rho(z^{*} - x_i^{[k+1]}) = \rho(z^{*} - \hat{z}_i^{[k]}) + (\hat{\lambda}_i^{[k]} - \lambda_i^*).
\end{equation}
Under Assumption~\ref{ass:02}, there exists $\sigma > 0$ such that for each subproblem:
$\nabla^2 f_i(x_i^{[k+1]}) + \rho I \succeq \sigma I \succeq 0,$
which implies $\left\|\nabla^2 f_i(x_i^{[k+1]}) + \rho I\right\| \succeq \sigma$. Consequently, we have:
\begin{equation}\label{eq: important inequality}\small
   \left \|\nabla f_i(z^{*}) - \nabla f_i(x_i^{[k+1]}) + \rho(z^{*} - x_i^{[k+1]})\right\| \geq \sigma \left\|z^* - x_i^{[k+1]}\right\|.
\end{equation}
Substituting~\eqref{eq: important equality} into~\eqref{eq: important inequality} gives:
\begin{equation}\label{eq: con_ALADIN_step1}
    \frac{\rho}{\sigma}\left\|\hat{z}_i^{[k]} - z^{*}\right\| + \frac{1}{\sigma}\left\|\hat{\lambda}_i^{[k]} - \lambda_i^*\right\| \geq \left\|x_i^{[k+1]} - z^*\right\|, \quad \forall i \in \mathcal{V}.
\end{equation}

The inner level of Algorithm~\ref{alg:Bilevel Consensus ALADIN} introduces an error term \(\tilde\epsilon = \left\|\hat{z}_i^{[k+1]} - z^{[k+1]}\right\|\), where \(\hat{z}_i^{[k+1]}\) is generated by \eqref{eq: local copy update} and \(z^{[k+1]}\) is generated by the iteration in \eqref{eq: QuDRC-ALADIN}.
To simplify the subsequent analysis, we introduce the following notation. Following \cite[Chapter 8, Chapter 12]{diehl2016lecture}, we introduce a sufficiently small parameter $\gamma > 0$ such that $\|B_i^{[k+1]} - \nabla^2 f_i(x_i^{[k+1]})\| < \gamma$, and define $\bar{\Lambda}^{[k+1]} = \tilde\epsilon \sum_{i=1}^N \Lambda_{\max}(B_i^{[k+1]})$ and $\bar{\Lambda}^{[k]} = \tilde\epsilon \sum_{i=1}^N \Lambda_{\max}(B_i^{[k]})$, where $\Lambda_{\max}(\cdot)$ denotes the largest eigenvalue of a matrix. Following the methodologies presented in \cite[equation (12)]{Engelmann2020} and \cite[equation (31)]{Du2025ACC}, we establish the following inequalities that hold prior to convergence:
\begin{equation}\label{eq: step2-linear2}\small
    \left\{
    \begin{aligned}
   N\tilde\epsilon + N\left\|{z}^{[k+1]} - z^*\right\| &\leq \gamma\sum_{i=1}^{N}\left\|x_i^{[k+1]} - z^*\right\|, \\
        \bar{\Lambda}^{[k+1]} + \sum_{i=1}^{N}\left\|{\lambda}_i^{[k+1]} - \lambda_i^*\right\| &\overset{\eqref{eq: closed form}}{\leq} \gamma\sum_{i=1}^{N}\left\|x_i^{[k+1]} - z^*\right\|.
    \end{aligned}
    \right.
\end{equation}
From \eqref{eq: step2-linear2}, we derive the following inequality, 
    \begin{align}\label{eq: local error analysis}\small
        &\left( \frac{\rho N}{\sigma}\left\|\hat{z}_i^{[k+1]} - z^*\right\| \hspace{-1mm}+\hspace{-1mm} \frac{1}{\sigma}\sum_{i=1}^{N}\left\|{\hat\lambda}_i^{[k+1]} - \lambda_i^*\right\| \hspace{-1mm}+\hspace{-1mm} \frac{\rho N\tilde\epsilon + \bar \Lambda^{[k+1]}}{\sigma}\hspace{-0.5mm}\right)\notag\\
        &\overset{\eqref{eq: step2-linear2}}{\leq} \frac{(\rho + 1)\gamma}{\sigma}\sum_{i=1}^{N}\left\|x_i^{[k+1]} - z^*\right\|.
    \end{align}
Combining \eqref{eq: con_ALADIN_step1} with \eqref{eq: local error analysis}, we derive the following inequality:
   {\small \begin{align}\label{eq: local error analysis last}\small
        &\hspace{-1mm}\left(\hspace{-1mm} \frac{\rho N}{\sigma}\left\|\hat z_i^{[k+1]} - z^*\right\| + \frac{1}{\sigma}\sum_{i=1}^{N}\left\|\hat\lambda_i^{[k+1]} - \lambda_i^*\right\| + \frac{\rho N\tilde\epsilon + \bar{\Lambda}^{[k+1]}}{\sigma} \hspace{-0.5mm}\right)\notag \\
        \leq & \frac{(\rho + 1)\gamma}{\sigma} \left( \frac{\rho N}{\sigma}\left\|\hat{z}_i^{[k]} - z^{*}\right\| + \frac{1}{\sigma}\sum_{i=1}^{N}\left\|\hat{\lambda}_i^{[k]} - \lambda_i^*\right\| \right) \\
        \leq & \frac{(\rho + 1)\gamma}{\sigma} \hspace{-1mm}\left( \hspace{-1mm}\frac{\rho N}{\sigma}\left\|\hat z_i^{[k]} - z^{*}\right\| \hspace{-1mm}+ \frac{1}{\sigma}\sum_{i=1}^{N}\left\|\hat \lambda_i^{[k]} - \lambda_i^*\right\| \hspace{-1mm}+ \frac{\rho N\tilde\epsilon + \bar{\Lambda}^{[k]}}{\sigma} \hspace{-1mm}\right).\notag
    \end{align}}
We observe that the left-hand side of \eqref{eq: local error analysis} exhibits an identical structure to the right-hand side of \eqref{eq: local error analysis last}, differing only by the contraction factor $\frac{(\rho + 1)\gamma}{\sigma} < 1$ for sufficiently small $\gamma$. This contraction property establishes the local linear convergence of Algorithm~\ref{alg:Bilevel Consensus ALADIN} prior to convergence.

Note that if Algorithm~\ref{alg:Bilevel Consensus ALADIN} reaches the precision threshold $\epsilon$, inequality \eqref{eq: step2-linear2} no longer holds. This indicates that  Algorithm~\ref{alg:BFGS ALADIN2} ceases to pursue higher accuracy and instead converges to a neighborhood of the optimal solution with radius $\epsilon$. This completes the local convergence analysis of Algorithm~\ref{alg:Bilevel Consensus ALADIN}.

\section{Proof of Algorithm~\ref{alg: Decentralized Reduced BFGS Consensus ALADIN}}\label{APP: local convergence error 2}
The proof of Algorithm~\ref{alg: Decentralized Reduced BFGS Consensus ALADIN} follows similarly to that in Appendix~\ref{APP: local convergence error 1}, with the only distinction being the definition of \(\tilde{\epsilon} = \left\| \hat{z}_i^{[k+1]}-z^{[k+1]} \right\|\), where \(\hat{z}_i^{[k+1]}\) is generated via \eqref{eq: primal_dual1} and \(z^{[k+1]}\) is generated by the inner level of Algorithm \ref{alg:Bilevel Consensus ALADIN}. Consequently, we focus solely on deriving an upper bound for \(\tilde{\epsilon}\).
Define $\bar{B} = \frac{1}{N} \sum_{i=1}^N B_i^{[k+1]}$ and $\bar{G} = \frac{1}{N} \sum_{i=1}^N g_i$. From \eqref{eq: primal_dual1}, we obtain the error bound:

\begin{equation}\label{eq: error bd}\small
    \begin{aligned}
      &  \left\| \hat{z}_i^{[k+1]} - z^{[k+1]} \right\| \\
        \overset{\eqref{eq: x AVG},\eqref{eq: g avg},\eqref{eq: primal_dual1},\eqref{eq: closed form}}{\leq} &
        \left\| \bar{x}_i^{[k+1]} - \bar{B}^{-1} \cdot \frac{1}{N} \sum_{i=1}^N B_i^{[k+1]} x_i^{[k+1]} \right\| \\
        &+ \left\| H_i^{[k+1]} l_i^{[k+1]} - \bar{B}^{-1} \bar{G} \right\|.
    \end{aligned}
\end{equation}
Here $\bar{x}_i^{[k+1]}$ is generated from \eqref{eq: x AVG}, implying that $\bar{x}_i^{[k+1]}$ is identical for all $i$.

The first term on the right-hand side of \eqref{eq: error bd} can be bounded as:
  \small{\begin{align}\label{eq: perturbation1}
        & \left\| \bar{x}_i^{[k+1]} - \frac{1}{N} \bar{B}^{-1} \sum_{i=1}^N B_i^{[k+1]} x_i^{[k+1]} \right\| \\
         \leq &\left\| \bar{B}^{-1} \right\| \left\| \frac{1}{N} \sum_{i=1}^N B_i^{[k+1]} \left( \bar{x}_i^{[k+1]} - x_i^{[k+1]} \right) \right\| \notag\\
        \overset{\eqref{eq: error}}{\leq}& \hspace{-0.5mm}\Lambda_{\text{max}}(\bar{B}^{-1}) \Lambda_{\text{max}}(B_i^{[k+1]}) 
        \hspace{-1mm}\left( \hspace{-0.5mm}\left\| \frac{1}{N} \sum_{i=1}^N \bar x_i^{[k+1]} - x_i^{[k+1]} \right\| + 2\sqrt{n}\Delta\hspace{-1mm} \right) \notag\\
          &\leq \tilde{\epsilon}_1 \notag.
    \end{align}}  

In the second term of \eqref{eq: error bd}, noting that $l_i^{[k+1]}$ is generated from \eqref{eq: g avg}, we have
\begin{equation}\label{eq: gradient est}\small
    \begin{split}
        & \left\| l_i^{[k+1]} - \frac{1}{N} \sum_{i=1}^N g_i \right\| \\
         \quad \overset{\eqref{eq: hessian estimation},\eqref{eq: error}}{\leq} &
        2\sqrt{n}\Delta + \left\| \frac{1}{N} \sum_{i=1}^N  \nabla f_i(\bar{x}_i^{[k+1]}) - \nabla f_i(x_i^{[k+1]})  \right\| \\
         \quad \overset{\eqref{eq: lip}}{\leq} &
        2\sqrt{n}\Delta + \frac{1}{N} \sum_{i=1}^N L_i\left\| \bar{x}_i^{[k+1]} - x_i^{[k+1]} \right\| \leq \hat{\epsilon}.
    \end{split}
\end{equation}
We further assume that for $H_i^{[k+1]}$ in \eqref{eq: hessian estimation},
\begin{equation}\label{eq: additional ass} \small
    \left\| \left(H_i^{[k+1]}\right)^{-1} - \bar{B} \right\| < \frac{1}{\left\| \left(H_i^{[k+1]}\right)^{-1} \right\|} = M_\epsilon,
\end{equation}
Applying perturbation theory for linear systems \cite[Theorem 7.29]{burden2010numerical} together with \eqref{eq: gradient est}, the second term in \eqref{eq: error bd} is bounded as follows,
\begin{equation}\label{eq: perturbation2}\small
    \begin{aligned}
        & \left\| H_i^{[k+1]} l_i^{[k+1]} - \bar{B}^{-1} \bar{G} \right\| \\
         \quad \overset{\eqref{eq: gradient est}}{\leq} &
        \frac{\mathrm{cond}(\bar{B}) \| z^{[k+1]} \|}{1 - \mathrm{cond}(\bar{B}) \frac{M_\epsilon}{\| \bar{B} \|}} 
        \left( \frac{\hat{\epsilon}}{\| \bar{G} \|} + \frac{M_\epsilon}{\| \bar{B} \|} \right) \leq \tilde{\epsilon}_2,
    \end{aligned}
\end{equation}
where $\mathrm{cond}(\cdot)$ denotes the condition number of a matrix.

Substituting the bounds from \eqref{eq: perturbation1} and \eqref{eq: perturbation2} into \eqref{eq: error bd}, we obtain the combined error estimate:
\begin{equation}\label{eq: last error bound}\small
\left\|\hat{z}_i^{[k+1]} - z^{[k+1]} \right\|\leq \tilde{\epsilon}_1 + \tilde{\epsilon}_2 = \tilde{\epsilon}.
\end{equation}
If the error bound $\tilde{\epsilon}$ from \eqref{eq: last error bound} satisfies the condition in \eqref{eq: step2-linear2}, then Algorithm~\ref{alg: Decentralized Reduced BFGS Consensus ALADIN} converges to a neighborhood of the optimal solution \(z^*\), thereby establishing local convergence; otherwise, convergence to this neighborhood is not guaranteed. This completes the proof of Algorithm~\ref{alg: Decentralized Reduced BFGS Consensus ALADIN}.


\section*{References}	\vspace{-6mm}
\bibliographystyle{IEEEtran}
\bibliography{paper}

\end{document}